\documentclass[12pt,draft]{article}
\usepackage{amsmath,amssymb,amsthm}
\usepackage{color}
\usepackage{bm} 
\usepackage{ascmac} 

\newtheorem{thm}{Theorem}
\newtheorem{crl}{Corollary}
\newtheorem{cnj}{Conjecture}
\newtheorem{prp}{Proposition}

\newtheorem{rmk}{Remark}
\newtheorem{dfn}{Definition}
\newtheorem{exm}{Example}

\newcommand{\R}{\Rightarrow}

\newcommand{\LR}{\Leftrightarrow}
\newcommand{\st}{\stackrel}
\newcommand{\ra}{\rightarrow}
\newcommand{\mt}{\mapsto}

\newcommand{\HFR}{\mathrm{Hom}(F,\mathbb R)}

\newcommand{\HKC}{\mathrm{Hom}(K,\mathbb C)}

\newcommand{\prn}{\mathbb R_{+}}

\newcommand{\nni}{\mathbb Z_{\geq 0}}
\newcommand{\ttt}{\,{}^t}

\renewcommand{\Re}{\mathrm{Re}}

\newcommand{\oq}{\overline{\mathbb Q}}

\makeatletter
\newcommand{\subjclass}[2][2010]{%
  \let\@oldtitle\@title%
  \gdef\@title{\@oldtitle\footnotetext{#1 \emph{Mathematics subject classification(s).} #2}}%
}
\newcommand{\keywords}[1]{%
  \let\@@oldtitle\@title%
  \gdef\@title{\@@oldtitle\footnotetext{\emph{Key words and phrases.} #1.}}%
}
\makeatother

\title{On a common refinement of Stark units and Gross-Stark units}
\author{Tomokazu Kashio\thanks{Tokyo University of Science, \texttt{kashio\_tomokazu@ma.noda.tus.ac.jp}}}
\subjclass{11G15, 11R27, 11R37, 11R42, 11R80, 11S40, 11S80, 14F30, 14K20, 14K22, 33B15.}
\keywords{Stark's conjectures, the Gross-Stark conjecture, CM-periods, ($p$-adic) multiple gamma functions, $p$-adic Hodge theory}

\begin{document}


\maketitle

\begin{abstract}
The purpose of this paper is to formulate and study a common refinement of a version of Stark's conjecture and its $p$-adic analogue, 
in terms of Fontaine's $p$-adic period ring and $p$-adic Hodge theory.
We construct period-ring-valued functions under a generalization of Yoshida's conjecture on the transcendental parts of CM-periods.
Then we conjecture a reciprocity law on their special values concerning the absolute Frobenius action.
We show that our conjecture implies a part of Stark's conjecture when the base field is an arbitrary real field and the splitting place is its real place.
It also implies a refinement of the Gross-Stark conjecture under a certain assumption.
When the base field is the rational number field, our conjecture follows from Coleman's formula on Fermat curves.
We also prove some partial results in other cases.
\end{abstract}


\section{Introduction}

Let $K/F$ be an abelian extension of number fields, $S$ a finite set of places.
We assume that
\begin{itemize}
\item[(a)] $S$ contains all infinite places of $F$ and ramified places in $K/F$.
\item[(b)] $S$ contains a distinguished place $v$ which splits completely in $K/F$.
We fix a place $w$ of $K$ lying above $v$.
\item[(c)] $|S|\geq 2$.
\end{itemize}
We define the partial zeta function associated to $S$ and $\sigma \in \mathrm{Gal}(K/F)$ by
\begin{align*}
\zeta_S(s,\sigma):=\sum_{\mathfrak a \subset \mathcal O_F,\ (\mathfrak a,S)=1,\ (\frac{K/F}{\mathfrak a})=\sigma} N\mathfrak a^{-s},
\end{align*}
where $\mathfrak a$ runs over all integral ideals of $F$ relatively prime to any prime ideals in $S$ whose images 
under the Artin symbol $(\frac{K/F}{*})$ is equal to $\sigma$.
The assumptions (b), (c) imply $\zeta_S(0,\sigma)=0$.
Then a version of Stark's conjectures \cite{St}, which is called the rank $1$ abelian Stark conjecture, states that 

\begin{cnj} \label{Sc}
There exists an element $\epsilon \in K^\times$ satisfying that 
\begin{enumerate}
\item If $|S| >2$, then $\epsilon$ is a $v$-unit.
If $S=\{v,v'\}$, then $\epsilon$ is an $S$-unit and $|\epsilon|_{w'}$ is constant whenever $w'$ is a place of $K$ lying above $v'$.
\item $\log |\epsilon^\sigma|_w=-W_K\zeta'_S(0,\sigma)$ $(\sigma \in \mathrm{Gal}(K/F))$.
\item $K(\epsilon^{\frac{1}{W_K}})/F$ is an abelian extension.
\end{enumerate}
Here $W_K$ denotes the number of roots of unity in $K^\times$.
\end{cnj}

In addition to (a), (b), (c) we assume that
\begin{itemize}
\item[(d)] $F$ is a totally real field and $v$ is a real place.
\end{itemize} 
Then Conjecture \ref{Sc} implies the following algebraicity property and reciprocity law:
\begin{align}
u(\sigma)&:=\exp(\zeta'_S(0,\sigma)) \in \oq^\times \quad (\sigma \in \mathrm{Gal}(\oq/F)), \notag \\
u(\sigma)^\tau&=\pm u(\tau|_K \sigma) \quad (\sigma,\tau \in \mathrm{Gal}(\oq/F)). \label{rlofSu}
\end{align}
In this paper, we call the element $u(\sigma)=\exp(\zeta'_S(0,\sigma))$ in the case (d) a Stark unit,
(In fact, the element $\epsilon$ in Conjecture \ref{Sc} is called a Stark unit in more general setting.)

On the other hand, Gross conjectured the following $p$-adic analogue:
we assume that, in addition to (a), (b), (c),   
\begin{itemize}
\item[(e)] $F$ is a totally real field, $K$ is a CM-field.
\item[(f)] $v$ is a finite place lying above a rational prime $p$.
\item[(g)] $S$ contains all places of $F$ lying above $p$.
\end{itemize}
Then there exists the $p$-adic interpolation function $\zeta_{S,p}(s,\sigma)$ of $\zeta_S(s,\sigma)$. 
Let $R:=S-\{v\}$. There obviously exist $W \in \mathbb N$ and a $v$-unit $\epsilon' \in K$ which satisfy
\begin{align*}
\log |\epsilon'^\sigma|_w=-W\zeta'_S(0,\sigma)=-W\zeta_R(0,\sigma)\log N v.
\end{align*}
Here Conjecture \ref{Sc} implies that we may take $W:=W_K$.
Gross' conjecture states that

\begin{cnj}[{\cite[Conjecture 3.13]{Gr2}}] \label{GSc}
\begin{align*}
\log_p N_{K_w/\mathbb Q_p}(\epsilon'^\sigma) =-W\zeta'_{S,p}(0,\sigma).
\end{align*}
\end{cnj}

In this paper, we call $\exp_p(\zeta'_{S,p}(0,\sigma))$, or its refinement given in \cite{KY1}, a Gross-Stark unit.
(Strictly speaking, a modified version of $\epsilon'$ is often called a Gross-Stark unit.)
We note that Dasgupta-Darmon-Pollack \cite{DDP} proved Conjecture \ref{GSc} under certain assumptions 
and Ventullo \cite{Ve} removed those assumptions.

The aim of this paper is to study a common refinement of (a part of, or, a generalization of) these two conjectures, which is formulated in terms of a period-ring-valued function.
Here the period ring is in the sense of $p$-adic Hodge theory.
The author believes that such an application of $p$-adic Hodge theory to this problem is a ``brand new approach''.
Let $C_\mathfrak f$ be the narrow ray class group modulo $\mathfrak f$ over a totally real field $F$.
We denote by $\infty_{\mathrm{id}},\mathfrak p_{\mathrm{id}}$ 
the real place, the finite place corresponding to fixed embeddings $\mathrm{id} \colon F \hookrightarrow \mathbb R,\mathbb C_p$ respectively.
When $\mathfrak p_\mathrm{id}$ divides $\mathfrak f$, we shall construct a period-ring-valued functions $\Gamma(c)$ for $c \in C_\mathfrak f$
which takes values in Fontaine's $p$-adic period ring, 
assuming Conjecture \ref{Yc} on the transcendental part of CM-periods.
We see that $\Gamma(c)$ is a ``common refinement'' of Stark units and Gross-Stark units in the following sense:
let $H$ be an abelian extension of $F$ with $\phi_H \colon C_\mathfrak f \ra \mathrm{Gal}(H/F)$ the Artin map.
By (\ref{rl}), (\ref{rGc}), Proposition \ref{prispr}, we see that 
\begin{align} 
&\text{the product }\prod_{c \in \phi_H^{-1}(\sigma)} \Gamma(c) \text{ equals a rational power of }
\begin{cases}
\text{ a Stark unit } \\
\text{ a Gross-Stark unit }
\end{cases} \notag \\
&\text{up to a root of unity, when the place }
\begin{cases}
\infty_\mathrm{id} \\
\mathfrak p_\mathrm{id}
\end{cases}
\text{splits completely in $H/F$}. \label{product}
\end{align}
We also construct another period-ring-valued function $\Gamma(c;D,\mathfrak a_c)$ when $\mathfrak p_\mathrm{id} \nmid \mathfrak f$. 
The relation between $\Gamma(c)$ and $\Gamma(c;D,\mathfrak a_c)$ is given in Proposition \ref{relofgc}-(iv).
Then we formulate Conjecture \ref{maincnj} on a reciprocity law on their special values, concerning the absolute Frobenius action on the $p$-adic period ring.
As evidence for conjectures, we show that 
\begin{itemize}
\item Conjectures \ref{Yc}, \ref{maincnj} imply the reciprocity law (\ref{rlofSu}) on Stark units, up to a root of unity (Theorem \ref{cns}).
\item A strong version (\ref{assump}) of Conjecture \ref{Yc} and Conjecture \ref{maincnj} imply certain refinements of Conjecture \ref{GSc} (Theorem \ref{cns2}).
\item When $F=\mathbb Q$, $p\neq 2$, Conjecture \ref{maincnj} holds true (Theorem \ref{cns3}).
\item When $F\neq \mathbb Q$, $p\neq 2$, we prove a partial result (Theorem \ref{partial}). 
In particular, if the narrow ray class field modulo $\mathfrak f$ is abelian over $\mathbb Q$ and 
if $\mathfrak p=(p) \nmid \mathfrak f$ or $\mathfrak p\mid \mathfrak f$, then Conjecture \ref{maincnj} holds true.
\end{itemize} 
Here, when $F=\mathbb Q$, Conjecture \ref{Yc} and (\ref{assump}) follow from well-known facts on periods of Fermat curves.
We note that each Stark unit and each Gross-Stark unit is expressed as a finite product of $\Gamma(c)$'s as in (\ref{product}).
Therefore Conjecture \ref{maincnj} on each $\Gamma(c),\Gamma(c;D,\mathfrak a_c)$ is a non-trivial refinement of (\ref{rlofSu}) and Conjecture \ref{GSc}. 

The rest of this paper is structured as follows.
In \S 2, we recall the definition and some properties of Yoshida's class invariant $X(c)$ and its $p$-adic analogue $X_p(c)$:
associated with each ideal class $c \in C_\mathfrak f$, Yoshida \cite{Yo} defined the invariant $X(c)$ 
as a finite sum of $\log$ of Barnes' multiple gamma functions with certain correction terms.
Yoshida and the author \cite{KY1} defined its $p$-adic analogue $X_p(c)$ in a quite similar manner.
Then both of Stark units and Gross-Stark units can be expressed in terms of the ``ratios'' $[\exp(X(c)):\exp_p(X_p(c))]$.
Such expressions were obtained in \cite{Ka4}.
In \S 3, we study some properties of CM-periods. 
Shimura's period symbol $p_K$ was defined by decomposing period integrals of abelian varieties with CM \cite{Shim}.
It can be expressed also as periods of motives associated with algebraic Hecke characters.
Then we introduce (a slight generalization of) Yoshida's conjecture on the transcendental part of CM-periods in \S 4.
Roughly speaking, for each $c$, the transcendental part of $\exp(X(c))$ is given as a product of rational powers of CM-periods (Conjecture \ref{Yc}).
In \S 5, we state the main results.
First we define a $p$-adic analogue of Shimura's period symbol by using comparison isomorphisms of $p$-adic Hodge theory, as usual.
Then, under Conjecture \ref{Yc}, 
we construct and study two types of period-ring-valued function $\Gamma(c),\Gamma(c;D,\mathfrak a_c)$, which take values in Fontaine's period ring.
We formulate Conjecture \ref{maincnj} on the absolute Frobenius action on their special values, 
which is a common refinement of the reciprocity law (\ref{rlofSu}) on Stark units, up to a root of unity, 
and Conjecture \ref{GSc} of Gross.
In \S 6, we show that Conjecture \ref{maincnj} holds true when $F=\mathbb Q$, $p\neq 2$,
which was essentially proved in \cite{Ka2} by using Coleman's results in \cite{Co}.
We note that it was proved by using properties of periods and $p$-adic periods of Fermat curves,
instead of using Euler's reflection formula $\Gamma(z)\Gamma(1-z)=\frac{\pi}{\sin(\pi z)}$.
Namely, we gave an alternative proof of a part of Stark's conjecture in the case $F=\mathbb Q$.
In \S 7, we also prove some partial results on Conjecture \ref{maincnj} in the case $F\neq \mathbb Q$.

\section{Yoshida's class invariant and its $p$-adic analogue}

In this section, we recall the definitions of Yoshida's class invariant $X(c)$ and its $p$-adic analogue $X_p(c)$.
In particular, we introduce a relation between $X(c),X_p(c)$ and Stark units, Gross-Stark units.

\subsection{Barnes' multiple gamma function and its $p$-adic analogue}

We denote by $\prn$ the set of all positive real numbers, by $\mathbb C_p$ the completion of 
algebraic closure $\overline{\mathbb Q_p}$ of $\mathbb Q_p$ with $p$ a rational prime, 
by $|z |_p=p^{-\mathrm{ord}_p(z)}$ the $p$-adic absolute value of $z\in \mathbb C_p$.
We regard any number field $k$ as a subfield of $\oq$, and fix embeddings $\oq \hookrightarrow \mathbb C,\mathbb C_p$.
In particular, $\mathrm{id} \colon k \hookrightarrow \mathbb C,\mathbb C_p$ are well-defined.

\begin{dfn} \label{mult}
\begin{enumerate}
\item Let $z \in \prn$, $\bm v \in \prn^r$.
Then Barnes' multiple zeta function $\zeta(s,\bm v,z)$ is defined as the meromorphic continuation of
\begin{align*}
\zeta(s,\bm v,z):=\sum_{\bm m \in \nni^r} (z+\bm m \ttt \bm v)^{-s} \quad (\Re(s)>r),
\end{align*}
which is analytic at $s=0$.
Here we put $(m_1,\dots,m_r) \ttt (v_1,\dots,v_r):=m_1v_1+\dots+m_rv_r$.
We define Barnes' multiple gamma function as
\begin{align*}
\Gamma(z,\bm v):=\exp\left(\frac{d}{ds}\zeta(s,\bm v,z)|_{s=0}\right).
\end{align*}
Note that we slightly modified this definition from Barnes' one.
\item Let $z \in \oq$, $\bm v =(v_1,\dots,v_r) \in \oq^r$.
We assume that 
\begin{enumerate}
\item $z \in \prn$, $\bm v \in \prn^r$ with respect to $\oq \hookrightarrow \mathbb C$.
\item $|z|_p > |v_1|_p,\dots,|v_r|_p$ with respect to $\oq \hookrightarrow \mathbb C_p$.
\end{enumerate}
Then the $p$-adic multiple zeta function $\zeta_p(s,\bm v,z)$ is characterized by
\begin{align*}
\zeta_p(-k,\bm v,z)=(p^{\mathrm{ord}_p(z)}\theta_p(z))^k\zeta_p(-k,\bm v,z) \quad (0\leq k \in \mathbb Z).
\end{align*}
Here $\mathrm{ord}_p\colon \mathbb C_p^\times \ra \mathbb Q$, $\theta_p\colon \mathbb C_p^\times \ra \mu_{(p)}$ 
{\rm(}$\mu_{(p)}$ denotes the group of all roots of unity of prime-to-$p$ order{\rm)} 
are unique group homomorphisms satisfying
\begin{align*}
|p^{-\mathrm{ord}_p(z)}\theta_p(z)^{-1}z -1|_p< 1 \quad (z\in \mathbb C_p^\times).
\end{align*}
We define the $p$-adic $\log$ multiple gamma function as 
\begin{align*}
L\Gamma_p(z,\bm v):=\frac{d}{ds}\zeta_p(s,\bm v,z)|_{s=0}.
\end{align*}
\item Let $\log_p$ be Iwasawa's $p$-adic $\log$ function:
\begin{align*}
\log_p(z):=\sum_{m=1}^\infty \frac{-(1-p^{-\mathrm{ord}_p(z)}\theta_p(z)^{-1}z)^m}{m} \quad (z\in \mathbb C_p^\times).
\end{align*}
We fix a group homomorphism
\begin{align*}
\exp_p\colon \mathbb C_p \ra \{z \in \mathbb C_p^\times  \mid |z|_p=1\}
\end{align*}
which coincides with the usual power series $\sum_{m=0}^\infty \frac{z^m}{m!}$ on a neighborhood of $0$.
In particular we see that $\exp_p(\log_p(z))\equiv p^{-\mathrm{ord}_p(z)}z$ modulo the roots of unity.
\end{enumerate}
\end{dfn}
For details concerning Barnes' functions or their $p$-adic analogues, see \cite[Chap.\ I, \S 1]{Yo} or \cite{Ka1} respectively.
For the latter one, we presented a short survey also in \cite[\S 2]{Ka4}.

\subsection{Definition of $X(c),X_p(c)$}

Let $F$ be a totally real field of degree $n$, $\HFR$ the set of all real embeddings of $F$.
We identify
\begin{align*}
F\otimes_\mathbb Q \mathbb R = \mathbb R^n,\ \sum_{i=1}^k a_i\otimes b_i \mt (\sum_{i=1}^k \iota(a_i)b_i)_{\iota \in \HFR}
\end{align*}
and define the totally positive part of $F\otimes_\mathbb Q \mathbb R$ as
\begin{align*}
(F \otimes_\mathbb Q \mathbb R)_+:=\prn^n.
\end{align*}
We denote by $\mathcal O$, $\mathcal O_+$, $E_+$
the ring of integers of $F$, the set of all totally positive elements in $\mathcal O$, the group of all totally positive units in $F$, respectively.
By \cite[Proposition 4]{Shin}, there always exists a fundamental domain of $(F\otimes_\mathbb Q \mathbb R)_+/ E_+$ 
which has good properties in the following sense.

\begin{dfn}
\begin{enumerate}
\item Let $v_1,\dots,v_r \in \mathcal O_+$ be linearly independent as elements in $\mathbb R^n$. 
Then we define the cone with basis $\bm v:=(v_1,\dots,v_r) $ by
\begin{align*}
C(\bm v):=\{\bm t \ttt \bm v \mid \bm t \in \prn^r\} \subset (F \otimes_\mathbb Q \mathbb R)_+.
\end{align*}
\item We call a fundamental domain $D$ of $(F\otimes_\mathbb Q \mathbb R)_+/ E_+$ a Shintani domain 
if it can be expressed as a finite disjoint union of cones with basis in $O_+$;
namely, we can write
\begin{align*}
D=\coprod_{j\in J} C(\bm v_j)\ (\bm v_j\in \mathcal O_+^{r(j)}),\quad (F \otimes_\mathbb Q \mathbb R)_+=\coprod_{\epsilon \in E_+} \epsilon D,
\end{align*}
where $J$ is a finite set of indices, $r(j)\in \mathbb N$.
\end{enumerate}
\end{dfn}

Let $\mathfrak f$ be an integral ideal of $F$.
We denote by $C_\mathfrak f$ the narrow ray class group modulo $\mathfrak f$.
We fix the following data:
\begin{itemize}
\item $D:=\coprod_{j \in J}C(\bm v_j)$ ($\bm v_j \in \mathcal O_+^{r(j)})$ is a Shintani domain.
\item For each $c \in C_\mathfrak f$, we take an integral ideal $\mathfrak a_c$ of $F$ satisfying $\mathfrak a_c \mathfrak f \in \pi(c)$,
where $\pi\colon C_\mathfrak f \ra C_{(1)}$ is the natural projection. 
\item For each prime ideal $\mathfrak q$ of $F$, we choose a generator $\pi_\mathfrak q \in \mathcal O_+$ of the principal ideal $\mathfrak q^{h_F^+}$,
where $h_F^+:=|C_{(1)}|$ is the narrow class number of $F$. Moreover we put
\begin{align*}
\pi_\mathfrak g:=\prod_{\mathfrak q \mid \mathfrak g} \pi_\mathfrak q^{\mathrm{ord}_\mathfrak q \mathfrak g}.
\end{align*}
\end{itemize}
We define
\begin{align*}
R(c,\bm v_j)&:=\{\bm x \in (\mathbb Q\cap (0,1])^{r(j)} \mid \mathcal O \supset (\bm x \ttt \bm v_j) \mathfrak a_c\mathfrak f \in c\},
\end{align*}
which is a finite set.
Then Yoshida's class invariant associated with the ideal class $c \in C_\mathfrak f$ is defined as the sum of three invariants
(for details, see \cite[Chap.\ III, (3.6)--(3.9)]{Yo}. The $W$-term is slightly modified in \cite[(4.3)]{KY1}, \cite[\S 2]{Ka3}):
\begin{align*}
X(c):=X(c;D,\mathfrak a_c)
:=G(c;D,\mathfrak a_c)+V(c;D,\mathfrak a_c)+W(c;D,\mathfrak a_c),
\end{align*}
where we put
\begin{align*}
G(c;D,\mathfrak a_c)&:=\left[\frac{d}{ds}\sum_{z \in (\mathfrak a_c\mathfrak f )^{-1} \cap D ,\ z \mathfrak a_c\mathfrak f \in c} z^{-s}\right]_{s=0} \\
&=\sum_{j \in J}\sum_{\bm x \in R(c,\bm v_j)}\log \Gamma(\bm x \ttt \bm v_j,\bm v_j), \\
W(c;D,\mathfrak a_c)&:=-\frac{\zeta(0,c)}{h_F^+} \log \pi_{\mathfrak a_c\mathfrak f}.
\end{align*}
We omit the detailed definition of $V(c;D,\mathfrak a_c)$ ($=V(c)$ in \cite[Chap.\ III, (3.7)]{Yo}).
Instead, we introduce \cite[Appendix I, Theorem]{KY2}:
there exist $a_i \in F^\times$, $\epsilon_i \in E_+$ satisfying 
\begin{align} \label{V}
V(c;D,\mathfrak a_c)=\sum_{i=1}^{n-1} a_i \log \epsilon_i. 
\end{align}

Next, let $p$ be a rational prime. 
We define the prime ideal 
\begin{align*}
\mathfrak p:=\mathfrak p_{\mathrm{id}}:=\{z \in \mathcal O \mid |z|_p<1\}
\end{align*}
corresponding to the $p$-adic topology on $F$ induced by $\mathrm{id}\colon F \hookrightarrow \mathbb C_p$.
Assume that $\mathfrak p \mid \mathfrak f$. Then we define \cite[\S 3]{KY1}
\begin{align*}
X_p(c)&:=X_p(c;D,\mathfrak a_c):=G_p(c;D,\mathfrak a_c)+V_p(c;D,\mathfrak a_c)+W_p(c;D,\mathfrak a_c), \\
G_p(c;D,\mathfrak a_c)&:=\sum_{j \in J}\sum_{\bm x \in R(c,\bm v_j)}L\Gamma_p(\bm x \ttt \bm v_j,\bm v_j), \\
W_p(c;D,\mathfrak a_c)&:=-\frac{\zeta(0,c)}{h_F^+} \log_p \pi_{\mathfrak a_c\mathfrak f}, \\
V_p(c;D,\mathfrak a_c)&:=\sum_{i=1}^{n-1} a_i \log_p \epsilon_i.
\end{align*}
Here we take the same elements $\pi_\mathfrak q,a_i,\epsilon_i$ as those for $V(c;D,\mathfrak a_c),W(c;D,\mathfrak a_c)$.
Note that the assumption $\mathfrak p \mid \mathfrak f$ implies that 
$(\bm x \ttt \bm v_j,\bm v_j)$ satisfies Definition \ref{mult}-(ii)-(b) for $\bm x \in R(c,\bm v_j)$.

Strictly speaking, $X(c;D,\mathfrak a_c),X_p(c;D,\mathfrak a_c)$ depend also on the choice of $\{\pi_\mathfrak q\}$,
not only on $c,D,\mathfrak a_c$.
However, by \cite[Lemma 3.11]{Ka3}, \cite[Lemma 4]{Ka4}, we see that 
\begin{align} \label{indep1}
\exp(X(c)) \in \mathbb C^\times/E_+^\mathbb Q, \ \exp_p(X_p(c)) \in \mathbb C_p^\times/E_+^\mathbb Q
\end{align}
do not depend on the choices of data $D,\mathfrak a_c,\pi_\mathfrak q$ (and $\exp_p$).
Here we put $E_+^\mathbb Q:=\{\alpha \in \oq^\times \mid \exists n \in \mathbb N\text{ s.t.\ }\alpha^n \in E_+\}$.
On the other hand, $X(c),X_p(c)$ highly depend on the choice of the embeddings $\mathrm{id}\colon F \hookrightarrow \mathbb R,\mathbb C_p$ respectively.

We shall use the following proposition repeatedly.

\begin{prp}[{\cite[(5.8), (5.9), and comments after (5.7)]{KY1}}] \label{575859}
Let $\mathfrak f$ be an integral ideal, $\mathfrak q$ a prime ideal, $\phi \colon C_{\mathfrak{fq}} \ra C_\mathfrak f$ the natural projection.
When $\mathfrak q \nmid \mathfrak f$, we denote by $[\mathfrak q] \in C_\mathfrak f$ the ideal class of $\mathfrak q$.
We fix $c \in C_\mathfrak f$ and assume that $\mathfrak q \mid \mathfrak a_c$. Then we have
\begin{align*}
&\sum_{\tilde c \in \phi^{-1}(c)} X(\tilde c;D,\mathfrak q^{-1}\mathfrak a_c) \\
&=
\begin{cases}
X(c;D,\mathfrak a_c) & (\mathfrak q \mid \mathfrak f), \\
X(c;D,\mathfrak a_c)-X([\mathfrak q]^{-1}c;D,\mathfrak q^{-1}\mathfrak a_c) +\frac{\zeta(0,[\mathfrak q]^{-1}c)}{h_F^+}\log \pi_\mathfrak q
& (\mathfrak q \nmid \mathfrak f).
\end{cases}
\end{align*}
Additionally assume that $\mathfrak p \mid \mathfrak f$. Then we have
\begin{align*}
&\sum_{\tilde c \in \phi^{-1}(c)} X_p(\tilde c;D,\mathfrak q^{-1}\mathfrak a_c) \\
&=
\begin{cases}
X_p(c;D,\mathfrak a_c) & (\mathfrak q \mid \mathfrak f), \\
X_p(c;D,\mathfrak a_c)-X_p([\mathfrak q]^{-1}c;D,\mathfrak q^{-1}\mathfrak a_c) +\frac{\zeta(0,[\mathfrak q]^{-1}c)}{h_F^+}\log_p \pi_\mathfrak q
& (\mathfrak q \nmid \mathfrak f).
\end{cases}
\end{align*}
\end{prp}

\subsection{Relation to Stark units, Gross-Stark units}

We introduce some results in \cite{Ka4}.
Let the notation be as in the previous subsection.
We assume that 
\begin{center}
$\mathfrak p$ divides $\mathfrak f$.
\end{center}
Then the ratio 
\begin{align*} 
[\exp(X(c;D,\mathfrak a_c)):\exp_p(X_p(c;D,\mathfrak a_c))] 
\in (\mathbb C^\times\times \mathbb C_p^\times)/(\mu_\infty\times \mu_\infty)E_+^\mathbb Q
\end{align*}
does not depend on the choices of $D,\mathfrak a_c,\pi_\mathfrak q$ by \cite[Lemmas 1, 4, Corollary 1]{Ka4}.
Here $\mu_\infty$ denotes the group of all roots of unity.
Namely we have 
\begin{align}
&\frac{\exp(X(c;D,\mathfrak a_c))}
{\exp(X(c;D',\mathfrak a_c'))},
\frac{\exp_p(X_p(c;D,\mathfrak a_c))}
{\exp_p(X_p(c;D',\mathfrak a_c'))} \in E_+^\mathbb Q, \notag \\
&\frac{\exp(X(c;D,\mathfrak a_c))}
{\exp(X(c;D',\mathfrak a_c'))} \equiv 
\frac{\exp_p(X_p(c;D,\mathfrak a_c))}
{\exp_p(X_p(c;D',\mathfrak a_c'))}  \mod \mu_\infty \label{indep2}
\end{align}
for other data $D',\mathfrak a_c',\pi_\mathfrak q'$.
In \cite[Remark 6]{Ka4} we rewrite the following two statements by using these ratios:
\begin{itemize}
\item The reciprocity law (\ref{rlofSu}) on Stark units, up to $\mu_\infty$.
\item A refinement \cite[Conjecture A$'$]{KY1} of Gross' conjecture (Conjecture \ref{GSc}).
\end{itemize}
We denote by $H_\mathfrak f$ the narrow ray class field modulo $\mathfrak f$.
For an intermediate field $H$ of $H_\mathfrak f/F$, we consider the Artin map
\begin{align*}
\phi_H\colon C_\mathfrak f \ra \mathrm{Gal}(H/F).
\end{align*}
First assume that the real place associated with $\mathrm{id}\colon F \hookrightarrow \mathbb R$ splits completely in $H/F$.
Then the reciprocity law (\ref{rlofSu}) implies that we have for $\sigma \in \mathrm{Gal}(H/F)$
\begin{align}
&\prod_{c \in \phi_H^{-1}(\sigma)}\frac{\exp(X(c))}{\exp_p(X_p(c))} 
\in \oq^\times, \notag \\
&\left(\prod_{c \in \phi_H^{-1}(\sigma)}\frac{\exp(X(c))}{\exp_p(X_p(c ))} \right)^{\tau} \notag \\
&\equiv \prod_{c \in \phi_H^{-1}(\tau|_H\sigma)}\frac{\exp(X(c))}{\exp_p(X_p(c ))} \mod \mu_\infty \quad (\tau \in \mathrm{Gal}(\oq/F)), \label{rl}
\end{align}
since \cite[Corollary 1-(ii)]{Ka4} states that 
\begin{align*}
\exp(\zeta'_S(0,\sigma)) \equiv \prod_{c \in \phi_H^{-1}(\sigma)}\frac{\exp(X(c))}{\exp_p(X_p(c))} \mod \mu_\infty
\end{align*}
for $S:=\{\text{infinite places}\}\cup \{\text{prime ideals dividing } \mathfrak f\}$.
Here the precise definition of $\prod_{c \in \phi_H^{-1}(\sigma)}\frac{\exp(X(c))}{\exp_p(X_p(c))}$ is as follows:
\cite[Theorem 4]{Ka4} implies that there exist $\epsilon(\sigma,D,\{\mathfrak a_c\}) \in E_+$, $N \in \mathbb N$, 
which depend on $D,\mathfrak a_c,\pi_\mathfrak q$,
but not on $\mathfrak p$ dividing $\mathfrak f$, satisfying 
\begin{align*}
\prod_{c \in \phi_H^{-1}(\sigma)}\exp_p(X_p(c ;D,\mathfrak a_c)) \equiv \epsilon(\sigma,D,\{\mathfrak a_c\})^\frac{1}{N} \mod \mu_\infty.
\end{align*}
Then we put 
\begin{align*}
\prod_{c \in \phi_H^{-1}(\sigma)}\frac{\exp(X(c))}{\exp_p(X_p(c ))} 
:=\frac{\prod_{c \in \phi_H^{-1}(\sigma)}\exp(X(c;D,\mathfrak a_c))}{\epsilon(\sigma,D,\{\mathfrak a_c\})^\frac{1}{N}} \in \mathbb C^\times
\end{align*}
by using the same data $D,\mathfrak a_c,\pi_\mathfrak q$.

Next assume that $H$ is a CM-field and $\mathfrak p$ splits completely in $H/F$. 
We consider the prime ideal $\mathfrak P_H$ corresponding to $\mathrm{id}\colon H \hookrightarrow \mathbb C_p$
and take a generator $\alpha_H$ of the principal ideal $\mathfrak P_H^{h_H}$ with $h_H$ the class number of $H$.
Let $\mathfrak f_0$ be the prime-to-$\mathfrak p$ part of $\mathfrak f$, $\phi_{H,0}\colon C_{\mathfrak f_0} \ra \mathrm{Gal}(H/F)$ the Artin map.
We put
\begin{align*}
\alpha:=\prod_{c \in C_{\mathfrak f_0}}\alpha_H^{\zeta(0,c^{-1})\phi_{H,0}(c)}.
\end{align*}
Yoshida and the author conjectured a refinement \cite[Conjecture A$'$]{KY1} of Conjecture \ref{GSc},
which is equivalent to the following statement by \cite[Proposition 6-(ii), Remark 6]{Ka4}.
\begin{align} \label{rGc0}
\left(\prod_{c \in \phi_H^{-1}(\sigma)}\frac{\exp(X(c))}{\exp_p(X_p(c))} \right)^{h_H}
\equiv \alpha^\sigma \mod \ker \log_p \quad (\sigma \in \mathrm{Gal}(H/F)).
\end{align}
In this case, we can show that $\prod_{c \in \phi_H^{-1}(\sigma)}\exp(X(c))\in \pi_\mathfrak p^r\cdot E_+^\mathbb Q$ with $r \in \mathbb Q$ as follows:
let $t:=\mathrm{ord}_\mathfrak p \mathfrak f$.
For simplicity, assume that $\mathfrak p^t \mid \mathfrak a_{c}$ for each $c \in C_{\mathfrak f}$.
Then, by using Proposition \ref{575859} $t$-times, we obtain
\begin{align*}
&\prod_{c \in \phi_H^{-1}(\sigma)}\exp(X(c;D,\mathfrak a_c)) \\
&=
\prod_{\overline c \in \phi_{H,0}^{-1}(\sigma)}\frac{\exp(X(\overline c;D,\mathfrak p^t\mathfrak a_c))}
{\exp(X([\mathfrak p]^{-1}\overline c;D,\mathfrak p^{t-1}\mathfrak a_c))}
\pi_\mathfrak p^{\frac{\sum_{\overline c \in \phi_{H,0}^{-1}(\sigma)}\zeta(0,[\mathfrak p]^{-1}\overline c)}{h_F^+}}.
\end{align*}
Here, in the right-hand side, we take $\mathfrak a_c$ associated with $c \in C_\mathfrak f$
whose image under the natural projection $C_\mathfrak f \ra C_{\mathfrak f_0}$ is equal to $\overline c$.
Moreover \cite[Lemma 1]{Ka4} implies that 
\begin{align*}
\prod_{c \in \phi_{H,0}^{-1}(\sigma)}\frac{\exp(X(c;D,\mathfrak a_c))}
{\exp(X([\mathfrak p]^{-1}c;D,\mathfrak p^{\mathrm{ord}_\mathfrak p \mathfrak f-1}\mathfrak a_c))}\in E_+^\mathbb Q.
\end{align*}
Hence we see that 
\begin{align*}
\prod_{c \in \phi_H^{-1}(\sigma)}\exp(X(c)) \equiv 
\pi_\mathfrak p^{\frac{\sum_{\overline c \in \phi_{H,0}^{-1}(\sigma)}\zeta(0,[\mathfrak p]^{-1}\overline c)}{h_F^+}} \mod E_+^\mathbb Q.
\end{align*}
In particular $\left(\prod_{c \in \phi_H^{-1}(\sigma)}\frac{\exp(X(c))}{\exp_p(X_p(c))} \right)^{h_H} \in \overline{\mathbb Q_p}^\times/\mu_\infty$ is well-defined. 
Furthermore, we reduce the ambiguity of (\ref{rGc0}):
since $\mathfrak p$ splits completely in $H/F$, 
we have $(\frac{H/F}{\mathfrak p})=\mathrm{id}_H$, $\frac{\mathrm{ord}_p \pi_\mathfrak p}{h_F^+}=\frac{\mathrm{ord}_p \alpha_H}{h_H}$. 
Therefore 
\begin{align*}
\mathrm{ord}_p(\alpha^\sigma)=\mathrm{ord}_p (\alpha_H)\sum_{c \in \phi_{H,0}^{-1}(\sigma)}\zeta(0,c)
\end{align*}
is equal to 
\begin{align*}
\mathrm{ord}_p\left(\prod_{c \in \phi_H^{-1}(\sigma)}\frac{\exp(X(c))}{\exp_p(X_p(c))}\right)^{h_H}
&=\frac{h_H\mathrm{ord}_p \pi_\mathfrak p}{h_F^+}\sum_{c \in \phi_{H,0}^{-1}(\sigma)}\zeta(0,[\mathfrak p]^{-1}c) \\
&=\frac{h_H\mathrm{ord}_p \pi_\mathfrak p}{h_F^+}\sum_{c \in \phi_{H,0}^{-1}(\sigma)}\zeta(0,c).
\end{align*}
Therefore (\ref{rGc0}) is equivalent to 
\begin{align} \label{rGc}
\left(\prod_{c \in \phi_H^{-1}(\sigma)}\frac{\exp(X(c))}{\exp_p(X_p(c))} \right)^{h_H} 
\equiv \alpha^\sigma \mod \mu_\infty \quad (\sigma \in \mathrm{Gal}(H/F))
\end{align}
since $\ker \log_p$ is generated by all rational powers of $p$ and $\mu_\infty$.

\section{Shimura's period symbol} 

In this section, we recall the definition and some properties of Shimura's period symbol $p_K$.
For later use, we give an explicit relation (Proposition \ref{pphi}) between $p_K$ and periods of motives associated with algebraic Hecke characters.

\subsection{Definition by abelian varieties with CM} 

For a number field $K$, we denote by $I_K$ (resp.\ $I_{K,\mathbb Q}$) the free $\mathbb Z$-module (the $\mathbb Q$-vector space)
formally generated by all complex embeddings of $K$:
\begin{align*}
I_K:=\bigoplus_{\iota \in \HKC} \mathbb Z \cdot \iota,\quad  
I_{K,\mathbb Q}:=\bigoplus_{\iota \in \HKC} \mathbb Q \cdot \iota.
\end{align*}
We regard each subset $S$ of $\HKC$ as an element in $I_K$ or $I_{K,\mathbb Q}$ by $S \leftrightarrow \sum_{\iota \in S} \iota$.
Since we fixed $\mathrm{id}\colon K \hookrightarrow \mathbb C$, 
we also regard $\mathrm{Aut}(K) \subset \HKC$ by $\sigma \leftrightarrow \mathrm{id}\circ \sigma$.
We denote the complex conjugation on $\mathbb C$ by $\rho$.
For $\Xi =\sum_{\iota \in \HKC} r_\iota \cdot \iota \in I_K$ or $\in I_{K,\mathbb Q}$, 
we put $\rho\circ \Xi:=\sum_{\iota \in \HKC} r_\iota (\rho\circ \iota)$.

\begin{dfn} \label{sps}
Let $K$ be a CM-field.
We define a $\mathbb Z$-bilinear map
\begin{align*}
p_K \colon I_K \times I_K \ra \mathbb C^\times/\oq^\times
\end{align*}
as follows.
\begin{enumerate}
\item We say $\Xi \subset \HKC$ is a CM-type of $K$ if $\Xi+\rho\circ \Xi=\HKC$.
For each CM-type $\Xi $ of $K$, 
we take an abelian variety $A_\Xi$ defined over $\oq$ with CM of type $(K,\Xi)$ and put
\begin{align} \label{period0}
p_K(\iota,\Xi):=
\begin{cases}
\pi^{-1} \int_{\gamma} \omega_{\iota} & (\iota \in \Xi), \\
\pi\left(\int_{\gamma} \omega_{\rho\circ\iota}\right)^{-1} & (\iota \in \rho \circ \Xi).
\end{cases}
\end{align}
Here $\omega_{\iota}$ denotes a non-zero holomorphic differential form of $A_\Xi$ defined over $\oq$ satisfying
\begin{align} \label{df}
k^*\omega_{\iota}=\iota(k)\omega_{\iota} \quad (k \in K)
\end{align}
concerning the action $k^*$ of $k$ through $K \cong \mathrm{End}(A)\otimes_\mathbb Z \mathbb Q$.
We take a closed path $\gamma$ of $A_\Xi(\mathbb C)$ which satisfies $\int_{\gamma} \omega_{\iota}\neq 0$.
\item Assume that $\Xi \in I_K$ satisfies $\Xi+\rho \circ \Xi=w \HKC$ with $w \in \mathbb Z$.
Then we take CM types $\Xi_j$ and integers $n_j$ satisfying $\Xi=\sum_{j=1}^l n_j\Xi_j$ and put
\begin{align*}
p_K(\iota,\Xi):=\prod_{j=1}^l p_K(\iota,\Xi_j)^{n_j} \quad (\iota \in \HKC),
\end{align*}
where $p_K(\iota,\Xi_j)$ is defined in {\rm (i)}.
\item For general $\Xi \in I_K$, we put
\begin{align} \label{nlz}
p_K(\iota,\Xi):=p_K(\iota,\Xi-\rho\circ \Xi)^{\frac{1}{2}}\quad (\iota \in \HKC),
\end{align}
where $p_K(\iota,\Xi-\rho\circ \Xi)$ is defined in {\rm (ii)}.
\item We extend $p_K(\iota,\Xi)$ of {\rm (iii)} linearly on the left component.
\end{enumerate}
\end{dfn}
In particular, by (\ref{nlz}), the symbol $p_K$ was ``normalized'' so that it satisfies 
\begin{align} \label{mr}
p_K(\iota,\Xi) p_K(\iota,\rho\circ \Xi) \equiv 1 \mod \overline{\mathbb Q}^\times.
\end{align}

By \cite[Theorems 32.2, 32.4, 32.5]{Shim}, 
we see that $p_K \colon I_K \times I_K \ra \mathbb C^\times/\oq^\times$
does not depend on the choices of $\Xi_j$, $n_j$, $A_\Xi$, $\gamma$, $\omega_{\iota}$.
We extend this linearly to
\begin{align*}
p_K \colon I_{K,\mathbb Q} \times I_{K,\mathbb Q} \ra \mathbb C^\times/\oq^\times.
\end{align*}
We consider the Lefschetz motive $\mathbb Q(-1)$.
Since a polarization of $A_\Xi$ induces a correspondence
\begin{align} \label{corresp}
H^1(A_\Xi) \times H^1(A_\Xi) \ra \mathbb Q(-1),
\end{align}
we have 
\begin{align*}
\int_{\gamma} \omega_{\iota} \int_{\gamma'} \omega_{\rho\circ\iota}\equiv 2\pi i \mod \oq^\times.
\end{align*}
Here, when $\iota \in \rho \circ \Xi$, we take a non-zero differential form $\omega_{\iota}$ of the second kind which is also characterized by (\ref{df}). 
For a proof, see \cite[Chap.\ 2, (1.5.4)]{Sc}.
Hence we may replace the definition (\ref{period0}) of $p_K(\iota,\Xi)$ with 
\begin{align} \label{period}
p_K(\iota,\Xi):=
\begin{cases}
(2\pi i)^{-1} \int_{\gamma} \omega_{\iota} & (\iota \in \Xi), \\
\int_{\gamma} \omega_{\iota} & (\iota \in \rho \circ \Xi).
\end{cases}
\end{align}

\subsection{Motives associated with algebraic Hecke characters} \label{Sps2}

We rewrite the definition of $p_K$, in terms of motives associated with algebraic Hecke characters.
The relation between periods of abelian variety with CM and those of such motives seems to be 
well-known for experts {\rm(\cite{Bl1}, \cite[Chap.\ 4, \S 1]{Sc})}.
Let $k$ be a number field, $\chi$ an algebraic Hecke character of $k$, 
$K$ a number field satisfying $\chi(\mathfrak a) \in K$ for any integral ideal $\mathfrak a$ of $k$ relatively prime to the conductor of $\chi$.
We say that a motive $M$ defined over $k$ with coefficients in $K$ is associated with $\chi$ if $\mathrm{rank}_K M=1$ and it satisfies 
\begin{align*}
L(M,s)=(L(s,\iota \circ \chi))_{\iota \in \HKC}.
\end{align*}
We denote such a motive by $M(\chi)$.
We consider the de Rham isomorphism between the Betti realization (associated with $\mathrm{id} \in \mathrm{Hom}(k,\mathbb C)$) and the de Rham realization:
\begin{align*}
I\colon H_\mathrm{B}(M(\chi)) \otimes_\mathbb Q\mathbb C \st{\cong}\ra H_\mathrm{dR}(M(\chi)) \otimes_k \mathbb C.
\end{align*}
We take the basis $c_b,c_{dr}$ of each space:
\begin{align*}
H_\mathrm{B}(M(\chi)) &=K c_b, \\
H_\mathrm{dR}(M(\chi)) &=(K\otimes k) c_{dr}.
\end{align*}
Then we define $\mathrm{Per}(\chi) \in (K\otimes_\mathbb Q \mathbb C)^\times$ by 
\begin{align*}
\mathrm{Per}(\chi)I(c_b\otimes 1)&:=c_{dr}\otimes 1.
\end{align*}
Let $\mathrm{Per}(\iota,\chi)$ be its $\iota \in \HKC$ component:
\begin{align*}
\mathrm{Per}(\chi)&=:(\mathrm{Per}(\iota,\chi))_{\iota \in \HKC} \quad \text{via }  
(K \otimes_\mathbb Q \mathbb C)^\times=\prod_{\iota \in \HKC}\mathbb C^\times.
\end{align*}
Since we have $M(\chi\chi')\cong M(\chi)\otimes_KM(\chi')$,
the symbol $\mathrm{Per}(\iota,\chi)$ also satisfies the ``linearity property'' {\rm(\cite[Chap.\ 2, (1.8.3)]{Sc})}:
\begin{align*}
\mathrm{Per}(\iota, \chi\chi') \equiv \mathrm{Per}(\iota, \chi)\mathrm{Per}(\iota, \chi') \mod \oq^\times.
\end{align*}
In particular, we see that $\mathrm{Per}(\iota,\chi) \bmod \oq^\times$ depends only on $\iota$ and the infinity type of $\chi$.
since we have
\begin{align} \label{PofA}
\mathrm{Per}(\iota,\chi) \in (K \otimes_\mathbb Q k^{\mathrm{ab}})^\times
\end{align}
for any Artin character $\chi$ by \cite[Chap.\ 2, (3.2.4), (3.2.3)]{Sc}, where $k^{\mathrm{ab}}$ denotes the maximal abelian extension of $k$.

Let $L$ be a number field satisfying that $L/\mathbb Q$ is normal and $K \subset L$.
For $\Xi \in I_K$, we consider the inflation $(L,\Xi^*_L)$ of the reflex type $(K^*,\Xi^*)$ of $(K,\Xi)$,
which provides the linear map  
\begin{align*}
I_K \ra I_L, \quad \Xi=\sum_{\iota \in \HKC} n_\iota \cdot \iota \mt \Xi^*_L=\sum_{\iota \in \mathrm{Hom}(L,\mathbb C)} n_{\iota|_K} \cdot \iota^{-1}.
\end{align*}
Here $\iota^{-1} \in \mathrm{Hom}(L,\mathbb C)$ is well-defined since $L/\mathbb Q$ is normal.
Let $A$ be an abelian variety defined over a number field $k$ with CM of type $(K,\Xi)$.
We consider the associated algebraic Hecke character $\chi_A$ of $k$.
We may assume that $k/\mathbb Q$ is normal and $K \subset k$. 
Then the infinity type of $\chi_A$ is equal to $\Xi_k^*$ 
by \cite[Theorem 19.8]{Shim}.
Moreover {\rm \cite[Chap.\ 2, (1.5.2)]{Sc}} states that 
\begin{align} \label{period2}
\int_\gamma \omega_\iota  \equiv \mathrm{Per}(\iota,\chi_A) \mod \oq^\times.
\end{align}

Let $\iota,\iota' \in \HKC$.
We take $\chi$ whose infinity type is equal to  
\begin{align} \label{itype}
(w \cdot \iota'- w \cdot \rho\circ \iota')^*_k
=w\sum_{\tilde{\iota'}\in \mathrm{Hom}(k,\mathbb C),\ \tilde{\iota'}|_K=\iota'}(\tilde{\iota'}^{-1}-\rho\circ \tilde{\iota'}^{-1})
\quad (w \in \mathbb N).
\end{align}
If we write $\iota'-\rho\circ \iota'=\sum_{j=1}^l n_j \Xi_j$ with integers $n_j$, CM-types $\Xi_j$, then we see that 
\begin{align} \label{sum}
\sum_{1\leq j \leq l,\ \iota \in \Xi_j} n_j=
\begin{cases}
1 & (\iota=\iota'), \\
-1 & (\iota=\rho\circ \iota'), \\
0 & (\text{otherwise}).
\end{cases}
\end{align}
Therefore we have by (\ref{period}), (\ref{period2})
\begin{align} \label{ptoP} 
p_K(\iota,\iota')\equiv
\begin{cases}
(2 \pi i)^{-\frac{1}{2}}\mathrm{Per}(\iota,\chi)^{\frac{1}{2w}}  & (\iota=\iota') \\
(2 \pi i)^{\frac{1}{2}}\mathrm{Per}(\iota,\chi)^{\frac{1}{2w}} & (\iota=\rho\circ \iota') \\
\mathrm{Per}(\iota,\chi)^{\frac{1}{2w}} & (\text{otherwise}) \\
\end{cases} \mod \oq^\times.
\end{align}
The above argument provides an alternative proof of well-definedness of $p_K$ modulo $\oq^\times$.
Moreover we can reduce the ambiguity:
let $\chi_{w\iota'}$ be an algebraic Hecke character of $\iota'(K)$ whose infinity type is equal to $w(\iota'^{-1}-\rho\circ \iota'^{-1})$ with $w \in \mathbb N$.
By taking a large enough $w$, we may assume that $\chi_{w\iota'}(\mathfrak a) \in K$.
Then the field of definition and the field of coefficients of $M(\chi_{w\iota'})$ are $\iota'(K)$ and $K$ respectively.
When $\iota'(K) \subset k$, the algebraic Hecke character $\chi_{w\iota'} \circ N_{k/\iota'(K)}$ of $k$ satisfies (\ref{itype}).
Moreover we have
\begin{align} \label{PtoP}
\mathrm{Per}(\iota,\chi_{w\iota'} \circ N_{k/\iota'(K)}) \equiv \mathrm{Per}(\iota,\chi_{w\iota'}) \mod \oq^\times
\end{align}
since $M(\chi_{w\iota'} \circ N_{k/\iota'(K)})=M(\chi_{w\iota'})\times_{\iota'(K)} k$.
The following proposition follows from (\ref{PofA}), (\ref{ptoP}), (\ref{PtoP}) and the linearity of symbols $p_K,\mathrm{Per}$.
\begin{prp} \label{pphi}
Let $\Xi:=\sum_{\iota \in \mathrm{Aut}(K)} r_\iota \cdot \iota \in I_{K,\mathbb Q}$ $(r_\iota \in \mathbb Q)$.
We note that $r_\iota=0$ when $\iota(K)\neq K$.
We take an integer $w$ and an algebraic Hecke character $\chi_{w\Xi}$ of $K$ satisfying that 
\begin{itemize}
\item The infinity type of $\chi_{w\Xi}$ is $\sum_{\iota \in \mathrm{Aut}(K)} w r_\iota (\iota^{-1}-\rho \circ \iota^{-1}) \in I_K$.
\item $\chi_{w\Xi}(\mathfrak a) \in K$ for $\mathfrak a$ relatively prime to the conductor of $\chi_{w\Xi}$.
\end{itemize}
We define $\mathrm{Per}(\chi_{w\Xi})=(\mathrm{Per}(\iota,\chi_{w\Xi}))_{\iota \in \HKC} \in (K\otimes_\mathbb Q \mathbb C)^\times$ by
\begin{align*}
\mathrm{Per}(\chi_{w\Xi}) I(c_b \otimes 1)=c_{dr} \otimes 1
\end{align*}
for a $K$-basis $c_b$ of $H_\mathrm{B}(M(\chi_{w\Xi}))$, a $K\otimes_\mathbb Q K$-basis $c_{dr}$ of $H_\mathrm{dR}(M(\chi_{w\Xi}))$.
\begin{enumerate}
\item For $\iota \in \mathrm{Aut}(K)$, 
\begin{align*}
\mathrm{Per}(\iota,\chi_{w\Xi})^{\frac{1}{2w}} \bmod ({K^{\mathrm{ab}}}^\times)^\mathbb Q
\end{align*}
depends only on $\iota,\Xi$.
Here $K^{\mathrm{ab}}$ denotes the maximal abelian extension of $K$ and we put 
$({K^{\mathrm{ab}}}^\times)^\mathbb Q :=\{\alpha \in \oq^\times \mid \exists n \in \mathbb N\text{ s.t.\ } \alpha^n \in K^{\mathrm{ab}}\}$.
\item We have
\begin{align*} 
p_K(\iota,\Xi)\equiv
(2 \pi i)^{\frac{r_{\rho\circ \iota}-r_{\iota}}{2}}\mathrm{Per}(\iota,\chi_{w\Xi})^{\frac{1}{2w}} \mod \oq^\times \quad (\iota \in \HKC).
\end{align*}
\end{enumerate}
\end{prp}

\section{Yoshida's conjecture on the transcendental part}

Let $F$ be a totally real field, $C_\mathfrak f$ the narrow ray class group modulo $\mathfrak f$.
We assume that $(F,\mathfrak f)\neq (\mathbb Q,(1))$, or equivalently, 
$\mathrm{ord}_{s=0}\sum_{c \in C_\mathfrak f}\zeta(s,c)=|\{\text{infinite places of $F$}\}\cup \{\text{prime ideals dividing $\mathfrak f$}\}|-1 \geq 1$.
We denote by $s_\iota \in C_\mathfrak f$ the ``complex conjugation'' at $\iota \in \HFR$:
namely we take an element $\nu_\iota \in \mathcal O$ satisfying
\begin{align*}
\nu_\iota \equiv 1 \bmod \mathfrak f,\ \iota(\nu_\iota)<0,\ \iota'(\nu_\iota)>0\ (\iota\neq \iota' \in \HFR),
\end{align*}
and put 
\begin{align*}
s_\iota:=[(\nu_\iota)] \in C_\mathfrak f.
\end{align*}
Strictly speaking, $s_\iota$ is the ideal class corresponding to the complex conjugation at $\iota \in \HFR$ 
via the Artin map (for a proof, see \cite[Chap.\ III, the first paragraph of \S 5.1]{Yo}).
We consider two subgroups $\langle s_\iota \rangle, \langle  s_\iota s_{\iota'} \rangle$ of 
$C_\mathfrak f$ generated by all complex conjugations $s_\iota$ ($\iota \in \HFR)$, 
all products $s_\iota s_{\iota'}$ of all pairs of complex conjugations ($\iota, \iota'  \in \HFR)$, respectively.

Yoshida formulated a conjecture \cite[Chap.\ III, Conjecture 3.9]{Yo} 
which expresses any value of Shimura's period symbol $p_K$ in terms of $\exp(X(c))$'s.
The following is its slight generalization.

\begin{cnj}[{\cite[Conjecture 5.5]{Ka3}}] \label{Yc}
If $\langle s_\iota \rangle \supsetneq \langle  s_\iota s_{\iota'} \rangle$, then we have for $c \in C_\mathfrak f$
\begin{align} \label{cnj1}
\exp(X(c))
\equiv 
\pi^{\zeta(0,c)}p_{H_{\mathrm{CM}}}(\mathrm{Art}(c),\sum_{c' \in C_{\mathfrak f}} \tfrac{\zeta(0,c')}{[H_\mathfrak f:H_{\mathrm{CM}}]} \mathrm{Art}(c'))
\mod \overline{\mathbb Q}^\times.
\end{align}
Here $H_{\mathrm{CM}}$ denotes the maximal CM-subfield of the narrow ray class field $H_{\mathfrak f}$ modulo $\mathfrak f$,
$\mathrm{Art} \colon C_\mathfrak f \ra \mathrm{Gal}(H_{\mathrm{CM}}/F)$ denotes the Artin map.
Otherwise we have
\begin{align} \label{cnj2}
\exp(X(c)) \in \overline{\mathbb Q}^\times.
\end{align}
\end{cnj} 

In order to explain the difference between \cite[Conjecture 5.5]{Ka3} and the above one, we prove the following proposition:
Proposition \ref{cm}-(i) implies that \cite[Conjecture 5.5]{Ka3} is equivalent to (\ref{cnj1}).
By Proposition \ref{cm}-(ii), we see that (\ref{cnj2}) is a natural complement of (\ref{cnj1}).

\begin{prp} \label{cm}
Let $H_\mathfrak f/F$, $s_\iota \in C_\mathfrak f$ $(\iota \in \HFR)$ be as above.
\begin{enumerate}
\item The following are equivalent: 
\begin{center}
$\langle s_\iota \rangle \supsetneq \langle s_\iota s_{\iota'} \rangle$ $\LR$ there exists a CM-subfield of $H_\mathfrak f$.
\end{center}
Furthermore, when $\langle s_\iota \rangle \supsetneq \langle  s_\iota s_{\iota'} \rangle$, 
the maximal CM-subfield $H_{\mathrm{CM}}$ corresponds to $\langle s_\iota s_{\iota'} \rangle$
via the class field theory.
\item When $\langle s_\iota \rangle = \langle  s_\iota s_{\iota'} \rangle$, we have $\zeta(0,c)=0$ for any $c \in C_\mathfrak f$.
\end{enumerate}
\end{prp}

\begin{proof}
(i) Note that the maximal totally real subfield of $H_{\mathfrak f}$ corresponds to $\langle s_\iota \rangle$ via the class field theory.
Let $H$ be the subfield of $H_\mathfrak f$ corresponding to $\langle s_\iota s_{\iota'} \rangle$. 
Then $H$ is the maximal subfield where all complex conjugations coincide (note that $s_\iota s_{\iota'}=s_\iota s_{\iota'}^{-1}$).
Therefore $\langle s_\iota \rangle \supsetneq \langle s_\iota s_{\iota'} \rangle$
means that the unique complex conjugation on $H$ is not trivial, that is, $H$ is the maximal CM-subfield. \\[5pt]
(ii) It suffices to show that any character $\chi$ of $C_\mathfrak f$ is not totally odd (i.e., $L(0,\chi)=0$).
Assume that $\chi(s_\iota)=-1$ for all $\iota \in \HFR$.
If $\langle s_\iota \rangle = \langle  s_\iota s_{\iota'} \rangle$, then we can write
$s_\iota=s_{\iota_1}\dotsm s_{\iota_{2k}}$,
so we have $\chi(s_\iota)=(-1)^{2k}=1$.
This is a contradiction. 
\end{proof}

When $F=\mathbb Q$, any $H_\mathfrak f$ is a CM-field, so (\ref{cnj1}) of Conjecture \ref{Yc} is equivalent to Conjecture \cite[Chap.\ III, Conjecture 3.9]{Yo}.
In this case, Yoshida \cite[Chap.\ III, Theorem 2.6]{Yo} proved that it follows from the following Rohrlich's formula in \cite[Theorem in Appendix]{Gr1}:
Consider $n$th Fermat curve $F_n \colon x^n+y^n=1$. Then each simple factor of its Jacobian variety has CM by $\mathbb Q(\zeta_n)$ 
and $\eta_{r,s}:=x^ry^{n-s} \frac{dx}{x}$ ($0<r,s<n$, $r+s\neq n$) are its differential forms of the second kind.
Then we have for any closed path $\gamma$ of $F_n(\mathbb C)$
\begin{align} \label{rh}
\int_\gamma \eta_{r,s} \equiv \frac{\Gamma(\frac{r}{n})\Gamma(\frac{s}{n})}{\Gamma(\frac{r+s}{n})} \mod \mathbb Q(\zeta_n)^\times.
\end{align}
On the other hand, (\ref{cnj2}) of Conjecture \ref{Yc} follows from Euler's reflection formula
\begin{align*}
\Gamma(s)\Gamma(1-s)=\frac{\pi}{\sin(\pi s)}.
\end{align*}
Yoshida also provided numerical examples \cite[Chap.\ III, \S 4]{Yo} in the case $F\neq \mathbb Q$.
We add one more example.

\begin{exm} \label{mawahc}
Let $K:=\mathbb Q(\sqrt{2\sqrt{5}-26})$, which is a non-abelian CM-field of degree $4$.
We define $\sigma \in \mathrm{Hom}(K,\mathbb C)$ by $\sigma(\sqrt{2\sqrt{5}-26}):= -\sqrt{-2\sqrt{5}-26}$ and denote the complex conjugation by $\rho$.
Then $\mathrm{Hom}(K,\mathbb C)=\{\mathrm{id},\rho,\sigma,\rho\circ\sigma\}$. 
The following curve is in the list of {\rm \cite{BS}}, whose Jacobian variety has CM by $K$ of CM-type $(K,\{\mathrm{id},\sigma\})$.
\begin{align*}
C\colon 
y^2=\tfrac{7+\sqrt{41}}{2}x^6+(-10-2\sqrt{41})x^5+10x^4+\tfrac{41+\sqrt{41}}{2}x^3+(3-2\sqrt{41})x^2+\tfrac{7-\sqrt{41}}{2}x+1.
\end{align*}
In fact, we see that
\begin{align*}
\omega_{\mathrm{id}}=\frac{2 dx}{y}+\frac{(\sqrt{5}-1) xdx}{y},\ \omega_{\sigma}=\frac{(-\sqrt{5}+\sqrt{41}) xdx}{y} 
\end{align*}
are holomorphic differential forms where $K$ acts via $\mathrm{id},\sigma$ respectively.
Numerically we have 
\begin{align*}
\int \omega_{\mathrm{id}}&=-0.4929421793\ldots-0.8116152991\ldots i, \\
\int \omega_\sigma&=-0.1395619319\ldots+0.1323795194\ldots i.
\end{align*}
Here we used Maple's command {\rm \texttt{periodmatrix}}.
Next, define $C'$ by  replacing $\sqrt{41}$ with $-\sqrt{41}$. We see that its Jacobian variety CM of CM-type $(K,\{\mathrm{id},\rho \circ \sigma\})$
and that 
\begin{align*}
\omega_{\mathrm{id}}':=\frac{2 dx}{y}+\frac{(\sqrt{5}-1) xdx}{y},\ \omega_{\rho \circ \sigma}':=\frac{(-\sqrt{5}+\sqrt{41}) xdx}{y}
\end{align*}
are its holomorphic differential forms where $K$ acts via $\mathrm{id},\rho \circ \sigma$ respectively.
In this case we have
\begin{align*}
\int \omega_{\mathrm{id}}'&=-0.4443866005\ldots-0.3099403507\ldots i, \\
\int \omega_{\rho \circ \sigma}'&=-2.0247186165\ldots+0.4533729269\ldots i.
\end{align*}
By {\rm (\ref{period0})}, {\rm (\ref{mr})}, we obtain
\begin{align*}
\pi p_K(\mathrm{id},\mathrm{id})p_K(\mathrm{id},\rho)^{-1}
&\equiv \pi p_K(\mathrm{id},\mathrm{id})p_K(\mathrm{id},\sigma)p_K(\mathrm{id},\mathrm{id})p_K(\mathrm{id},\rho\circ \sigma) \\
&\equiv \pi^{-1} \int \omega_{\mathrm{id}}\int \omega_{\mathrm{id}}' \mod \overline{\mathbb Q}^\times.
\end{align*}
On the other hand, $K$ is abelian over $F:=\mathbb Q(\sqrt{5})$ with the conductor $\mathfrak f:=(\frac{13-\sqrt{5}}{2})$.
We easily see that  
\begin{align*}
&C_\mathfrak f=\{c_1,c_2\} \text{ with } c_1:=[(1)],c_2:=[(3)], \\ 
&C_\mathfrak f\cong \mathrm{Gal}(K/F),\ c_1 \leftrightarrow \mathrm{id},\ c_2 \leftrightarrow \rho, \\ 
&\zeta(0,c_1)=1,\ \zeta(0,c_2)=-1. 
\end{align*}
Then {\rm Conjecture \ref{Yc}} in this case states that
\begin{align*}
\exp(X(c_1)) \equiv \pi^{-1} \int \omega_{\mathrm{id}}\int \omega_{\mathrm{id}}'\mod \overline{\mathbb Q}^\times.
\end{align*}
In fact, we found numerically that  
\begin{align} \label{num}
\pi^{-1} \int \omega_{\mathrm{id}}\int \omega_{\mathrm{id}}'
=\exp(X(c_1;D,\mathfrak a_{c_1}))(\tfrac{\sqrt{5}-1}{2})^{\frac{14}{41}}\tfrac{\sqrt{-8\sqrt{5}+20+(\sqrt{5}+15)\sqrt{2\sqrt{5}-26}}}{80}
\end{align}
for $D:=C(1,\frac{3+\sqrt 5}{2})\coprod C(1)$, $\mathfrak a_{c_1}:=\mathfrak f$.
\end{exm}

\section{A reciprocity law}

In this section, we define two types of period-ring-valued functions and conjecture a reciprocity law on their special values.
We also study a relation between Stark's conjecture, Gross' conjecture and this reciprocity law.

\subsection{A $p$-adic analogue of Shimura's period symbol} 
$p$-adic analogues of $p_K$ were also studied in \cite{dS}, \cite{Gi}, \cite{KY2}.
For an abelian variety $A$ defined over $\oq$, we consider the following comparison isomorphisms 
in the sense of $p$-adic Hodge theory developed by many mathematicians (for example, \cite{Fo1}, \cite{Fo2}, \cite{Fa}, \cite{Ts}):
let $B_{\mathrm{dR}},B_{\mathrm{cris}}$ be Fontaine's $p$-adic period rings.
We have 
\begin{align*}
H_\mathrm{B}^1(A(\mathbb C),\mathbb Q)\otimes_\mathbb Q \mathbb Q_p &\cong H_{p,\textrm{\'et}}^1(A,\mathbb Q_p), \\
H_{p,\textrm{\'et}}^1(A,\mathbb Q_p)\otimes_{\mathbb Q_p} B_{\mathrm{dR}} &\cong H_\mathrm{dR}^1(A,\oq)\otimes_{\oq} B_{\mathrm{dR}}.
\end{align*}
Here we denote by $H_B$, $H_{p,\textrm{\'et}}$ and $H_\mathrm{dR}$, the singular cohomology group, the $p$-adic \'etale cohomology group
and the de Rham cohomology group respectively.
We consider the composite map
\begin{align} \label{Ip}
H_\mathrm{B}^1(A(\mathbb C),\mathbb Q)\otimes_\mathbb Q B_{\mathrm{dR}} \cong H_\mathrm{dR}^1(A,\oq)\otimes_{\oq} B_{\mathrm{dR}},
\end{align}
which is a $p$-adic analogue of the de Rham isomorphism
\begin{align} \label{I}
H_\mathrm{B}^1(A(\mathbb C),\mathbb Q)\otimes_\mathbb Q \mathbb C \cong H_\mathrm{dR}^1(A,\oq)\otimes_{\oq} \mathbb C.
\end{align}
Additionally we have the duality
\begin{align*}
H_1(A(\mathbb C),\mathbb Q) \times H_\mathrm{B}^1(A(\mathbb C),\mathbb Q) \ra \mathbb Q
\end{align*}
between the singular cohomology group $H_\mathrm{B}^1$ and the singular homology group $H_1$.
Hence we obtain the $p$-adic period integral 
\begin{align*}
H_1(A(\mathbb C),\mathbb Q) \times H_\mathrm{dR}^1(A,\oq) \ra B_{\mathrm{dR}}, \ 
(\gamma,\omega) \mt \int_{\gamma,p} \omega,
\end{align*}
which is a $p$-adic analogue of usual period integral 
\begin{align*}
H_1(A(\mathbb C),\mathbb Q) \times H_\mathrm{dR}^1(A,\oq) \ra \mathbb C, \ 
(\gamma,\omega) \mt \int_\gamma \omega.
\end{align*}
If $A$ has CM, then it has potentially good reduction, so we see that 
\begin{align*}
\int_{\gamma,p} \omega \in B_{\mathrm{cris}}\overline{\mathbb Q_p}.
\end{align*}
Here $B_{\mathrm{cris}}\overline{\mathbb Q_p}$ denotes the composite ring of $B_{\mathrm{cris}}$ and $\overline{\mathbb Q_p}$ in $B_{\mathrm{dR}}$.
Let $\mathbb Q(-1)$ be the Lefschetz motive.
Similarly to (\ref{I}), (\ref{Ip}), we obtain 
\begin{align*}
&I\colon H_\mathrm{B}(\mathbb Q(-1))\otimes_\mathbb Q  \mathbb C \st{\cong}\ra H_\mathrm{dR}(\mathbb Q(-1)) \otimes_\mathbb Q \mathbb C, \\
&I_p\colon H_\mathrm{B}(\mathbb Q(-1))\otimes_\mathbb Q B_{\mathrm{dR}} \st{\cong}\ra H_\mathrm{dR}^1(\mathbb Q(-1))\otimes_\mathbb Q B_{\mathrm{dR}}.
\end{align*}
Taking a non-zero element $c \in H_\mathrm{B}(\mathbb Q(-1))$, 
we define the $p$-adic counterpart $(2\pi i)_p \in B_\mathrm{cris}$ of $2\pi i$ by
\begin{align*}
I_p(c \otimes (2\pi i)_p):= I(c \otimes 2\pi i) \quad (\in  H_\mathrm{dR}(\mathbb Q(-1))).
\end{align*}
Then, for a CM-field $K$, we define a $\mathbb Q$-bilinear map 
\begin{align*}
p_{K,p} \colon I_{K,\mathbb Q} \times I_{K,\mathbb Q} \ra B_{\mathrm{dR}}^\times/\oq^\times
\end{align*}
by replacing $p_K(\iota,\Xi)$ of {\rm (\ref{period})} with 
\begin{align*} 
p_{K,p}(\iota,\Xi):=
\begin{cases}
(2\pi i)_p^{-1} \int_{\gamma,p} \omega_{\iota} & (\iota \in \Xi), \\
\int_{\gamma,p} \omega_{\iota} & (\iota \in \rho \circ \Xi).
\end{cases}
\end{align*}
On the other hand, by the results in \cite{Bl2}, we obtain the following comparison isomorphisms of motives:
let $M(\chi)$ be a motive associated with an algebraic Hecke character $\chi$ of $k$, $K$ the field of coefficients of $M(\chi)$.
We denote by $\mathfrak P$ the prime ideal of $k$ corresponding to the $p$-adic topology, by $k_\mathfrak P$ the $\mathfrak P$-adic completion of $k$, 
$W(k_\mathfrak P)$ the ring of Witt vectors over the residue field of $k_\mathfrak P$. 
We denote by $H_{\mathrm{cris}}$ the crystalline realization.
Then we have
\begin{align*}
H_{\mathrm{B}}(M(\chi))\otimes_\mathbb Q \mathbb Q_p &\cong H_{p,\textrm{\'et}}(M(\chi)), \\
H_{p,\textrm{\'et}}(M(\chi)) \otimes_{\mathbb Q_p} B_{\mathrm{cris}} &\cong H_{\mathrm{cris}}(M(\chi)) \otimes_{W(k_\mathfrak P)} B_{\mathrm{cris}}, \\
H_{\mathrm{cris}}(M(\chi)) \otimes_{W(k_\mathfrak P)} k_\mathfrak P &\cong H_{\mathrm{dR}}(M(\chi)) \otimes_k k_\mathfrak P.
\end{align*}
For details, see \cite[\S 2]{KY2}.
Combining these isomorphisms, we obtain
\begin{align} \label{ci}
H_{\mathrm{B}}(M(\chi))\otimes_\mathbb Q B_{\mathrm{cris}}\overline{\mathbb Q_p} 
\cong H_{\mathrm{dR}}(M(\chi)) \otimes_k B_{\mathrm{cris}}\overline{\mathbb Q_p}.
\end{align}
Taking the same basis $c_b \in H_{\mathrm{B}}(M(\chi))$, $c_{dr} \in H_{\mathrm{dR}}(M(\chi))$ as those in \S \ref{Sps2}, 
we define $\mathrm{Per}_p(\chi)=(\mathrm{Per}_p(\iota,\chi))_{\iota} \in K\otimes_\mathbb Q B_{\mathrm{cris}}\overline{\mathbb Q_p}
=\oplus_{\iota }B_{\mathrm{cris}}\overline{\mathbb Q_p}$ by
\begin{align*}
\mathrm{Per}_p(\chi)I(c_b\otimes 1):=c_{dr}\otimes 1.
\end{align*}
By replacing 
$2\pi i, p_K,\mathrm{Per}$ with $(2\pi i)_p, p_{K,p},\mathrm{Per}_p$ respectively, 
we obtain a $p$-adic counterpart of (\ref{ptoP}).
In particular we see that 
\begin{align*}
p_{K,p} \colon I_{K,\mathbb Q} \times I_{K,\mathbb Q} \ra B_\mathrm{dR}^\times/ \oq^\times
\end{align*} 
is well-defined. We also generalize some results in \cite{Shim} as follows.

\begin{prp} \label{prpofp}
Let the notation be as in the definition of $p_K,p_{K,p}$.
We denote by $\mu_\infty$ the group of all roots of unity.
\begin{enumerate}
\item For each CM type $\Xi$, the ratio
\begin{align*}
[\int_{\gamma} \omega_{\iota}:\int_{\gamma,p} \omega_{\iota}] \in (\mathbb C^\times \times B_\mathrm{dR}^\times)/\oq^\times
\end{align*}
does not depend on the choices of $A_\Xi,\gamma,\omega_{\iota}$.
In particular, the bilinear map
\begin{align*}
[p_K:p_{K,p}]\colon I_{K,\mathbb Q}\times I_{K,\mathbb Q} \ra (\mathbb C^\times \times B_\mathrm{dR}^\times)/(\mu_\infty\times\mu_\infty)\oq^\times
\end{align*}
is well-defined.
Here we assume that we take the same $\Xi_j,n_j,A_\Xi,\gamma,\omega_{\iota}$ 
when we define $p_K(\Xi,\Xi'),p_{K,p}(\Xi,\Xi')$ for each $\Xi,\Xi' \in I_{K,\mathbb Q}$.
\item For $\iota,\iota' \in \HKC$, we have
\begin{align*}
[p_K:p_{K,p}](\iota+\rho\circ \iota,\iota') \equiv [p_K:p_{K,p}](\iota,\iota'+\rho\circ \iota') \equiv [1:1] \mod \mu_\infty.
\end{align*}
Here $[p_K:p_{K,p}](\Xi,\Xi')$ means that $[p_K(\Xi,\Xi'):p_{K,p}(\Xi,\Xi')]$,
$[a:b]\equiv [c:d] \bmod \mu_\infty$ means that $\frac{a}{c}\frac{d}{b} \in \mu_\infty$.
\item Let $\phi\colon K \ra \phi(K)$ be an isomorphism between CM-fields. Then for $\iota,\iota' \in \HKC$ we have
\begin{align*}
[p_K:p_{K,p}](\iota,\iota') \equiv [p_{\phi(K)}:p_{\phi(K),p}](\iota \circ\phi^{-1},\iota' \circ\phi^{-1}) \mod \mu_\infty.
\end{align*}
\item Let $K\subset \tilde K$ be an extension of CM-fields. Then for $\tilde \iota \in \mathrm{Hom}(\tilde K,\mathbb C)$, $\iota' \in \HKC$ we have
\begin{align*}
&[p_K:p_{K,p}](\tilde \iota|_K,\iota') 
\equiv [p_{\tilde K}:p_{\tilde K,p}](\tilde \iota,\sum_{\tilde \iota' \in \mathrm{Hom}(K',\mathbb C),\ \tilde{\iota'}|_K=\iota'}\tilde \iota') \mod \mu_\infty, \\
&[p_K:p_{K,p}](\iota', \tilde \iota|_K) 
\equiv [p_{\tilde K}:p_{\tilde K,p}](\sum_{\tilde \iota' \in \mathrm{Hom}(K',\mathbb C),\ \tilde{\iota'}|_K=\iota'}\tilde \iota',\tilde \iota) \mod \mu_\infty.
\end{align*}
\end{enumerate}
\end{prp}

\begin{proof}
(i) The first assertion follows by the same argument used in the proof of \cite[Theorem 32.2-(1)]{Shim}.
The ambiguity up to $\mu_\infty$ occurs when we take a rational power of 
$\int_{\gamma} \omega_{\iota}$ or $\int_{\gamma,p} \omega_{\iota}$. \\[5pt]
(ii) The correspondence (\ref{corresp}) implies 
\begin{align} \label{2pii}
[\int_{\gamma} \omega_{\iota} \int_{\gamma'} \omega_{\rho\circ\iota}:\int_{\gamma,p} \omega_{\iota} \int_{\gamma',p} \omega_{\rho\circ\iota}]
=[2\pi i: (2\pi i)_p].
\end{align}
Hence, by (\ref{sum}) (and $\sum_{j=1}^l n_j=0$), we can write
\begin{align*}
&[p_K:p_{K,p}](\iota+\rho\circ \iota,\iota') \\
&\equiv 
\prod_{j=1}^l  [p_K:p_{K,p}](\iota+\rho\circ \iota,\Xi_j)^{n_j} \\
&\equiv \prod_{\substack{1\leq j \leq l\\ \iota \in \Xi_j}}[(2\pi i)^{-1}:(2\pi i)_p^{-1}]^{n_j} 
\prod_{\substack{1\leq j \leq l\\ \rho \circ \iota \in \Xi_j}}[(2\pi i)^{-1}:(2\pi i)_p^{-1}]^{n_j} 
\prod_{j=1}^l  [2\pi i: (2\pi i)_p]^{n_j} \\
& \equiv [1:1] \mod \mu_\infty.
\end{align*}
The assertion $[p_K:p_{K,p}](\iota,\iota'+\rho\circ \iota') \equiv [1:1] \bmod \mu_\infty$ is trivial by definition. \\[5pt]
(iii), (iv) follow by the same argument used in the proof of \cite[Theorem 32.5-(3), (4), (5)]{Shim}.
\end{proof}

By replacing $2\pi i, p_K,\mathrm{Per}$ with $(2\pi i)_p, p_{K,p},\mathrm{Per}_p$ respectively, 
we also obtain a $p$-adic counterpart of Proposition \ref{pphi}.
Moreover we have  
\begin{align} 
&[p_K:p_{K,p}](\iota,\Xi) \notag \\
&\equiv
[(2 \pi i)^{\frac{r_{\rho\circ \iota}-r_{\iota}}{2}}\mathrm{Per}(\iota,\chi_{w\Xi})^{\frac{1}{2w}}:
(2 \pi i)_p^{\frac{r_{\rho\circ \iota}-r_{\iota}}{2}}\mathrm{Per}_p(\iota,\chi_{w\Xi})^{\frac{1}{2w}}] \mod \mu_\infty, \label{pptoPP}
\end{align}
where $\Xi=\sum_{\iota \in \mathrm{Aut}(K)} r_\iota \cdot \iota$ and $\chi_{w\Xi}$ are as in Proposition \ref{pphi}.

For later use, we prepare a proposition concerning the absolute Frobenius action on $p$-adic periods.

\begin{dfn} \label{frb}
Let $k/\mathbb Q_p$ be a finite extension, $\mathfrak P$ the prime ideal of $k$.
We denote by $k^{\mathrm{ur}} \subset \overline{\mathbb Q_p}$ the maximal unramified extension of $k$, 
by $\mathrm{Frob}_\mathfrak P \in \mathrm{Gal}(k^{\mathrm{ur}}/k)$ the Frobenius automorphism.
The Weil group $W_k$ is defined as the group of all $\tau \in \mathrm{Gal}(\overline{\mathbb Q_p} /k)$ satisfying 
\begin{center}
$\tau|_{k^{\mathrm{ur}}}=\mathrm{Frob}_\mathfrak P^{\deg_\mathfrak P \tau}$ with $\deg_\mathfrak P \tau \in \mathbb Z$.
\end{center}
Associated to $\tau \in W_k$, we define a $\tau$-semilinear action $\Phi_\tau$ 
on $B_{\mathrm{cris}}\overline{\mathbb Q_p}\cong B_{\mathrm{cris}}\otimes_{\mathbb Q_p^{\mathrm{ur}}} \overline{\mathbb Q_p}$ by
\begin{align*}
\Phi_\tau:=\Phi_{\mathrm{cris}}^{\deg_\mathfrak P \tau \deg \mathfrak P} \otimes \tau,
\end{align*}
where $\Phi_{\mathrm{cris}}$ denotes the absolute Frobenius action on $B_{\mathrm{cris}}$.
\end{dfn}

\begin{prp}[{\cite[(3.5)]{KY2}}] \label{defP}
Let $\chi$ be an algebraic Hecke character of a number field $k$.
We take $M(\chi),K$ as above, and define $\mathrm{Per}_p(\iota,\chi) \in B_{\mathrm{cris}}\overline{\mathbb Q_p}$ $(\iota \in \HKC)$ 
for any basis $c_b \in H_{\mathrm{B}}(M(\chi))$, $c_{dr} \in H_{\mathrm{dR}}(M(\chi))$.
Let $\mathfrak P$ be the prime ideal corresponding to the $p$-adic topology on $k$.
Assume that $\chi$ is unramified at $\mathfrak P$.
Then we have for $\tau \in W_{k_{\mathfrak P}}$ with $\deg_{\mathfrak P} \tau=1$
\begin{align*}
\Phi_\tau(\mathrm{Per}_p(\iota,\chi))
=\iota(\chi(\mathfrak P))\mathrm{Per}_p(\iota,\chi).
\end{align*}
Additionally, let $\Xi_\chi$ be the infinity type of $\chi$, $\pi_\mathfrak P$ a generator of $\mathfrak P^{h_k}$ with $h_k$ the class number.
Then we have
\begin{align*}
\Phi_\tau(\mathrm{Per}_p(\iota,\chi))
\equiv \iota(\Xi_\chi(\pi_\mathfrak P))^{\frac{1}{h_k}}\mathrm{Per}_p(\iota,\chi) \mod \mu_\infty.
\end{align*}
Here for $\Xi=\sum_{\iota \in \mathrm{Hom}(k,\mathbb C)} n_\iota \cdot \iota \in I_k$, $\alpha \in k$, 
we put $\Xi(\alpha):=\prod_{\iota} \iota(\alpha)^{n_\iota}$ as usual.
\end{prp}

\begin{proof}
Although the assertion follows from \cite[(3.5)]{KY2}, we provide a proof since there are some differences of notation.
We can rewrite (\ref{ci}) in terms of the comparison isomorphism between the $p$-adic \'etale cohomology group and the crystalline cohomology group:
\begin{align*} 
H_{p,\textrm{\'et}}(M(\chi)) \otimes_{\mathbb Q_p} B_{\mathrm{cris}}\overline{\mathbb Q_p} 
&\cong (H_{\mathrm{cris}}(M(\chi)) \otimes_{W(k_\mathfrak P)}k_\mathfrak P) 
\otimes_{k_\mathfrak P} B_{\mathrm{cris}}\overline{\mathbb Q_p} \\
\mathrm{Per}_p(\chi)(c_b\otimes 1) \ \quad &\leftrightarrow \qquad\qquad\qquad c_{dr}\otimes 1,
\end{align*}
which is consistent with the absolute Frobenius action:
\begin{align*}
1\otimes \Phi_\tau \leftrightarrow \Phi_{\mathrm{cris}}^{\deg_\mathfrak P \tau \deg \mathfrak P} \otimes 1 \otimes \Phi_\tau
\quad (\tau\in W_{k_\mathfrak P}).
\end{align*}
Here we regard 
\begin{align*}
&c_b \in H_{\mathrm{B}}(M(\chi)) \subset H_{\mathrm{B}}(M(\chi)) \otimes_\mathbb Q \mathbb Q_p \cong H_{p,\textrm{\'et}}(M(\chi)), \\ 
&c_{dr} \in H_{\mathrm{dR}}(M(\chi)) \subset H_{\mathrm{dR}}(M(\chi)) \otimes_k k_\mathfrak P \cong
H_{\mathrm{cris}}(M(\chi)) \otimes_{W(k_\mathfrak P)}k_\mathfrak P.
\end{align*}
We note that
\begin{itemize}
\item The absolute Frobenius acts trivially on $H_{p,\textrm{\'et}}(M(\chi))$.
\item $\Phi_{\mathrm{cris}}^{\deg \mathfrak P}$ acts on $H_{\mathrm{cris}}(M(\chi))$ as $\chi(\mathfrak P) \otimes 1$ (\cite[(2.4)]{KY2}).
\end{itemize}
Then we have for $\tau \in W_{k_\mathfrak P}$ with $\deg_{\mathfrak P} \tau=1$
\begin{align*}
\Phi_\tau(\mathrm{Per}_p(\chi)(c_b\otimes 1))&=
\Phi_\tau(\mathrm{Per}_p(\chi))(c_b\otimes 1), \\ 
\Phi_\tau(c_{dr}\otimes 1)&=(\chi(\mathfrak P) \otimes 1)(c_{dr}\otimes 1).
\end{align*}
Then the first assertion follows.
Let $l\in \mathbb N$ satisfy $\pi_\mathfrak P^{l} \equiv 1$ modulo the conductor of $\chi$.
Then we see that 
\begin{align*}
\chi(\mathfrak P)\equiv \chi(\mathfrak P^{h_kl})^{\frac{1}{h_kl}}
\equiv \chi(\pi_\mathfrak P^l)^{\frac{1}{h_kl}}
\equiv \Xi_\chi(\pi_\mathfrak P^l)^{\frac{1}{h_kl}}
\equiv \Xi_\chi(\pi_\mathfrak P)^{\frac{1}{h_k}}\mod \mu_\infty.
\end{align*}
Hence the second assertion is clear.
\end{proof}

\begin{rmk}
Both definitions, by abelian varieties and motives, have their good points:
the latter definition simplifies the proof of well-definedness of $p_K,p_{K,p}$.
By the former definition, we easily see that 
\begin{align*}
[p_K:p_{K,p}]\colon I_{K}\times I_{K,0} \ra (\mathbb C^\times \times B_\mathrm{dR}^\times)/\oq^\times
\end{align*}
is well-defined for $I_{K,0}:=\{\sum_{\iota \in \HKC} n_\iota \cdot \iota \in I_K \mid n_\iota+n_{\rho\circ \iota}$ is constant $\}$.
\end{rmk}

\subsection{A period-ring-valued function}

We construct two kinds of period-valued functions $\Gamma(c),\Gamma(c;D,\mathfrak a_c)$
under assuming Conjecture \ref{Yc} holds true in this subsection.
By (\ref{rl}), (\ref{rGc}) and Proposition \ref{prispr} below, we see that $\Gamma(c)$ is a ``common refinement'' of Stark units and Gross-Stark units.
Let the notation be as in the previous sections:
let $F$ be a totally real field, $C_\mathfrak f$ the narrow ray class group modulo $\mathfrak f$, 
$s_\iota \in C_\mathfrak f$ the complex conjugation at $\iota \in \HFR$, $D$ a Shintani domain of $F$.
For each $c \in C_\mathfrak f$, we take an integral ideal $\mathfrak a_c$ of $F$ satisfying $\mathfrak a_c\mathfrak f \in \pi(c)$ 
with $\pi\colon C_\mathfrak f \ra C_{(1)}$.
For each prime ideal $\mathfrak q$, we fix a totally positive generator $\pi_\mathfrak q$ of $\mathfrak q^{h_F^+}$ with $h_F^+$ the narrow class number.
We denote by $H_{\mathrm{CM}}$ the maximal CM-subfield of the narrow ray class field $H_\mathfrak f$ modulo $\mathfrak f$, if it exists.
For an intermediate field $H$ of $H_\mathfrak f/F$, we denote by $\phi_H\colon C_\mathfrak f \ra \mathrm{Gal}(H/F)$ the Artin map.
We denote by $\infty_\mathrm{id}$, $\mathfrak p:=\mathfrak p_{\mathrm{id}}$ the places corresponding to 
$\mathrm{id}\colon F \hookrightarrow \mathbb R$, $\mathrm{id}\colon F \hookrightarrow \mathbb C_p$ respectively.

\begin{dfn} \label{maindfn}
We put for $*=\emptyset,p$
\begin{align*}
&\mathrm{P}_*(c) \\
&:=
\begin{cases}
(2\pi i)_*^{\zeta(0,c)}
p_{H_{\mathrm{CM}},*}(\mathrm{Art}(c),\sum_{c' \in C_{\mathfrak f}} \tfrac{\zeta(0,c')}{[H_\mathfrak f:H_{\mathrm{CM}}]} \mathrm{Art}(c')) 
& (\langle s_\iota \rangle \supsetneq \langle  s_\iota s_{\iota'}\rangle), \\
1 & (\langle s_\iota \rangle = \langle  s_\iota s_{\iota'}\rangle).
\end{cases}
\end{align*}
Under {\rm Conjecture \ref{Yc}}, we define two types of period-ring-values functions 
which take values in 
\begin{align*}
(B_{\mathrm{cris}}\overline{\mathbb Q_p}-\{0\})^\mathbb Q
:=\{\alpha \in B_{\mathrm{dR}}^\times \mid \exists n \in \mathbb N\text{ s.t.\ }\alpha^n \in B_{\mathrm{cris}}\overline{\mathbb Q_p}\},
\end{align*}
well-defined up to $\mu_\infty$, as follows.
\begin{enumerate}
\item Assume that $\mathfrak p \nmid \mathfrak f$. Then we put
\begin{align*}
\Gamma(c;D,\mathfrak a_c):=\frac{\exp(X(c;D,\mathfrak a_c))}{\mathrm{P}(c)}
\mathrm{P}_p(c).
\end{align*}
Strictly speaking, $\Gamma(c;D,\mathfrak a_c) \bmod \mu_\infty$ depends also on the choices of $\pi_\mathfrak q$'s.
\item Assume that $\mathfrak p\mid \mathfrak f$. 
Then we put
\begin{align*}
\Gamma(c):=\frac{\exp(X(c;D,\mathfrak a_c))}
{\mathrm{P}(c)}
\frac{\mathrm{P}_p(c)}
{\exp_p(X_p(c;D,\mathfrak a_c))}.
\end{align*}
\end{enumerate}
\end{dfn}

\begin{prp} \label{prispr}
We put $S:=\{$infinite places of $F\}\cup\{$prime ideals dividing $\mathfrak f\}$.
Assume that 
\begin{align*}
\zeta_S(0,\sigma)=0\text{ for all }\sigma \in \mathrm{Gal}(H/F).
\end{align*}
Then the ``period part'' of $\prod_{c\in \phi_H^{-1}(\sigma)}\Gamma(c)$ or $\prod_{c\in \phi_H^{-1}(\sigma)}\Gamma(c;D,\mathfrak a_c)$ 
becomes trivial: namely we have
\begin{align*}
\prod_{c\in \phi_H^{-1}(\sigma)}[\mathrm{P}(c):\mathrm{P}_p(c)]\equiv [1:1] \mod \mu_\infty \quad (\sigma \in \mathrm{Gal}(H/F)).
\end{align*}
In particular, when $\mathfrak p \mid \mathfrak f$ and when $\infty_\mathrm{id}$ or $\mathfrak p$ splits completely in $H/F$, we have
\begin{align*}
\prod_{c\in \phi_H^{-1}(\sigma)}\Gamma(c) \equiv \prod_{c \in \phi_H^{-1}(\sigma)}\frac{\exp(X(c))}{\exp_p(X_p(c))}  \mod \mu_\infty.
\end{align*}
\end{prp}

\begin{proof}
First we note that $\sum_{c\in \phi_H^{-1}(\sigma)}\zeta(s,c)=\zeta_S(s,\sigma)$.
We assume that $\langle s_\iota \rangle \supsetneq \langle  s_\iota s_{\iota'}\rangle$ (otherwise there is nothing to prove since $\mathrm{P}_*(c)=1$).
Then by Proposition \ref{prpofp}-(iii) we have 
\begin{align*}
&\prod_{c\in \phi_H^{-1}(\sigma)}\mathrm{P}_*(c) \\
&\equiv (2\pi i)_*^{\zeta_S(0,\sigma)}
p_{H_{\mathrm{CM}},*}(\sum_{c\in \phi_H^{-1}(\sigma)}\mathrm{Art}(c),
\sum_{c' \in C_{\mathfrak f}} \tfrac{\zeta(0,c')}{[H_\mathfrak f:H_{\mathrm{CM}}]} \mathrm{Art}(c')) \\
&\equiv (2\pi i)_*^{\zeta_S(0,\sigma)}
p_{H_{\mathrm{CM}},*}(\mathrm{id},
\sum_{c\in \phi_H^{-1}(\sigma)}\sum_{c' \in C_{\mathfrak f}} \tfrac{\zeta(0,c')}{[H_\mathfrak f:H_{\mathrm{CM}}]} \mathrm{Art}(c^{-1}c')) \\
&\equiv (2\pi i)_*^{\zeta_S(0,\sigma)}
p_{H_{\mathrm{CM}},*}(\mathrm{id},
\sum_{c' \in C_{\mathfrak f}} \tfrac{\sum_{c\in \phi_H^{-1}(\sigma)} \zeta(0,cc')}{[H_\mathfrak f:H_{\mathrm{CM}}]} \mathrm{Art}(c')) \mod \mu_\infty.
\end{align*}
By assumption, we have $\zeta_S(0,\sigma)=0$, $\sum_{c\in \phi_H^{-1}(\sigma)} \zeta(0,cc')=\zeta_S(0,\phi_H(c')\sigma)=0$.
Hence the assertion is clear.
\end{proof}

\begin{prp} \label{relofgc}
Let $\mathfrak f$ be an integral ideal, $\mathfrak q$ a prime ideal, $\phi\colon C_\mathfrak {fq} \ra C_\mathfrak f$ the natural projection.
When $\mathfrak q \nmid \mathfrak f$, $[\mathfrak q]$ denotes the ideal class of $\mathfrak q$ in $C_\mathfrak f$.
Let $c \in C_\mathfrak f$.
\begin{enumerate}
\item We have
\begin{align*}
&\prod_{\tilde c \in \phi^{-1}(c)} [\mathrm{P}(\tilde c):\mathrm{P}_p(\tilde c)] \\
&\equiv 
\begin{cases}
[\mathrm{P}(c)\mathrm{P}([\mathfrak q]^{-1}c)^{-1}:\mathrm{P}_p(c)\mathrm{P}_p([\mathfrak q]^{-1}c)^{-1}] & (\mathfrak q \nmid \mathfrak f) \\
[\mathrm{P}(c):\mathrm{P}_p(c)] & (\mathfrak q \mid \mathfrak f)
\end{cases}
\mod \mu_\infty.
\end{align*}
\item When $\mathfrak p \mid \mathfrak f$, we have
\begin{align*}
\prod_{\tilde c \in \phi^{-1}(c)} \Gamma(\tilde c) \equiv 
\begin{cases}
\Gamma(c)\Gamma([\mathfrak q]c)^{-1} & (\mathfrak q \nmid \mathfrak f) \\
\Gamma(c) & (\mathfrak q \mid \mathfrak f) 
\end{cases}
\mod \mu_\infty.
\end{align*}
\item When $\mathfrak p \nmid \mathfrak f$, $\mathfrak q \neq \mathfrak p$, $\mathfrak q \mid \mathfrak a_c$, we have
\begin{align*}
&\prod_{\tilde c \in \phi^{-1}(c)} \Gamma(\tilde c;D,\mathfrak q^{-1}\mathfrak a_c) \\
&\equiv 
\begin{cases}
\pi_\mathfrak q^{\frac{\zeta(0,[\mathfrak q]^{-1}c)}{h_F^+}} 
\Gamma(c;D,\mathfrak a_c)\Gamma([\mathfrak q]^{-1}c;D,\mathfrak q^{-1}\mathfrak a_c)^{-1} 
& (\mathfrak q \nmid \mathfrak f) \\
\Gamma(c;D,\mathfrak a_c) & (\mathfrak q \mid \mathfrak f) 
\end{cases}
\mod \mu_\infty.
\end{align*}
\item When $\mathfrak p \nmid \mathfrak f$, $\mathfrak q=\mathfrak p$, $\mathfrak q \mid \mathfrak a_c$, we have
\begin{align*}
&\prod_{\tilde c \in \phi^{-1}(c)}  \Gamma(\tilde c) \exp_p(X_p(\tilde c;D,\mathfrak q^{-1}\mathfrak a_c)) \\
&\equiv \pi_\mathfrak q^{\frac{\zeta(0,[\mathfrak q]^{-1}c)}{h_F^+}} 
\Gamma(c;D,\mathfrak a_c)\Gamma([\mathfrak q]^{-1}c;D,\mathfrak q^{-1}\mathfrak a_c)^{-1} \mod \mu_\infty.
\end{align*}
\end{enumerate}
\end{prp}

\begin{proof}
(i) In this proof, $H_{\mathrm{CM},\mathfrak f}$ denotes the subfield of $H_{\mathfrak f}$ 
corresponding to $\langle s_\iota s_{\iota'}\rangle$ via the class field theory,
even if it is not a CM-field.
Let $\mathrm{Art}_\mathfrak f \colon C_\mathfrak f \ra \mathrm{Gal}(H_{\mathrm{CM},\mathfrak f}/F)$ be the Artin map.
We easily see that 
\begin{align} \label{note1}
\sum_{\tilde c \in \phi^{-1}(c)} \mathrm{Art}_\mathfrak{fq}(\tilde c)=
\tfrac{[H_{\mathfrak{fq}}:H_{\mathrm{CM},\mathfrak{fq}}]}{[H_{\mathfrak{f}}:H_{\mathrm{CM},\mathfrak{f}}]}\mathrm{Inf}(\mathrm{Art}_\mathfrak{f}(c)),
\end{align}
where $\mathrm{Inf}$ denotes the inflation map:
\begin{align*}
\mathrm{Inf}
\colon I_{H_{\mathrm{CM},\mathfrak{f}},\mathbb Q} \ra I_{H_{\mathrm{CM},\mathfrak{fq}},\mathbb Q}, \ 
\sum r_\iota \cdot \iota \mt \sum r_{\tilde\iota|_{H_{\mathrm{CM},\mathfrak{f}}}} \cdot \tilde \iota.
\end{align*}
Since we can write 
\begin{align} \label{relofzeta}
\sum_{\tilde c \in \phi^{-1}(c)} \zeta(0,\tilde c)=
\begin{cases}
\zeta(0,c)-\zeta(0,[\mathfrak q]^{-1}c) & (\mathfrak q \nmid \mathfrak f), \\
\zeta(0,c) & (\mathfrak q \mid \mathfrak f),
\end{cases}
\end{align}
we see that 
\begin{align} \label{note2}
\sum_{\tilde c \in C_{\mathfrak{fq}}}\zeta(0,\tilde c)\mathrm{Art}_\mathfrak{fq} (\tilde c)|_{H_{\mathrm{CM},\mathfrak{f}}}
=(1-\mathrm{Art}_\mathfrak{f} ([\mathfrak q]))\sum_{c \in C_{\mathfrak{f}}}\zeta(0,c)\mathrm{Art}_\mathfrak{f} (c),
\end{align}
where we drop the term $(1-\mathrm{Art}_\mathfrak{f} ([\mathfrak q]))$ when $\mathfrak q \mid \mathfrak f$.
First assume that $\langle s_\iota \rangle \supsetneq \langle  s_\iota s_{\iota'}\rangle$ in $C_\mathfrak f$, 
i.e., both $H_{\mathrm{CM},\mathfrak f}$ and $H_{\mathrm{CM},\mathfrak{fq}}$ are CM-fields.
By Proposition \ref{prpofp}-(iii), (iv), (\ref{note1}), (\ref{note2})
we can write for $*=\emptyset, p$
\begin{align*}
&\prod_{\tilde c \in \phi^{-1}(c)} (2\pi i)_*^{-\zeta(0,c)}  \mathrm{P}_{*}(\tilde c) \\
&\equiv p_{H_{\mathrm{CM},\mathfrak{fq}},*}
(\tfrac{[H_{\mathfrak{fq}}:H_{\mathrm{CM},\mathfrak{fq}}]}{[H_{\mathfrak{f}}:H_{\mathrm{CM},\mathfrak{f}}]}\mathrm{Inf}(\mathrm{Art}_\mathfrak{f}(c)),
\sum_{\tilde c' \in C_{\mathfrak{fq}}}\tfrac{\zeta(0,\tilde c')}{[H_\mathfrak{fq}:H_{\mathrm{CM},\mathfrak{fq}}]}\mathrm{Art}_\mathfrak{fq} (\tilde c')) \\
&\equiv p_{H_{\mathrm{CM},\mathfrak{f}},*}(\mathrm{Art}_\mathfrak{f}(c),
(1-\mathrm{Art}_\mathfrak{f} ([\mathfrak q])) \sum_{c' \in C_{\mathfrak{f}}}\tfrac{\zeta(0,c')}{[H_{\mathfrak{f}}:H_{\mathrm{CM},\mathfrak{f}}]}
\mathrm{Art}_\mathfrak{f} (c')) \\
&\equiv p_{H_{\mathrm{CM},\mathfrak{f}},*}
((1-\mathrm{Art}_\mathfrak{f} ([\mathfrak q])^{-1}) \mathrm{Art}_\mathfrak{f}(c),
\sum_{c' \in C_{\mathfrak{f}}}\tfrac{\zeta(0,c')}{[H_{\mathfrak{f}}:H_{\mathrm{CM},\mathfrak{f}}]}\mathrm{Art}_\mathfrak{f} (c')).
\end{align*}
We drop the terms $(1-\mathrm{Art}_\mathfrak{f} ([\mathfrak q])),(1-\mathrm{Art}_\mathfrak{f} ([\mathfrak q])^{-1})$ when $\mathfrak q \mid \mathfrak f$.
Then the assertion follows.
Next, assume that $\langle s_\iota \rangle \supsetneq \langle s_\iota s_{\iota'}\rangle$ in $C_\mathfrak{fq}$
and that $\langle s_\iota \rangle = \langle  s_\iota s_{\iota'}\rangle$ in $C_\mathfrak{f}$.
It suffices to show that 
\begin{align*}
\prod_{\tilde c \in \phi^{-1}(c)} [\mathrm{P}(\tilde c):\mathrm{P}_p(\tilde c)] \equiv [1:1] \mod \mu_\infty,
\end{align*}
which follows from Proposition \ref{prispr}.
When $\langle s_\iota \rangle = \langle  s_\iota s_{\iota'}\rangle$ in $C_\mathfrak{fq}$, the assertion is trivial. \\[5pt]
(ii), (iii), (iv) follow from (i),  (\ref {relofzeta}), Proposition \ref{575859}. 
\end{proof}

\subsection{Conjecture on a reciprocity law}

Let $F,C_\mathfrak f,D,\mathfrak a_c,\pi_\mathfrak q,\mathfrak p$ be as in the previous subsection.
We consider the $\tau$-semilinear action $\Phi_\tau$ on $B_{\mathrm{cris}}\overline{\mathbb Q_p}$ associated with $\tau \in W_{F_\mathfrak p}$ 
defined in Definition \ref{frb} and 
$(B_{\mathrm{cris}}\overline{\mathbb Q_p}-\{0\})^\mathbb Q/\mu_\infty$-valued functions $\Gamma(c;D,\mathfrak a_c),\Gamma(c)$ defined in Definition \ref{maindfn}. 
In particular
\begin{align*}
\Phi_\tau(\Gamma(c;D,\mathfrak a_c)),\Phi_\tau(\Gamma(c)) \in (B_{\mathrm{cris}}\overline{\mathbb Q_p}-\{0\})^\mathbb Q/\mu_\infty
\end{align*}
are well-defined.

\begin{cnj} \label{maincnj}
Let $\tau \in W_{F_\mathfrak p}$. We assume that {\rm Conjecture \ref{Yc}} holds true for $F$.
\begin{enumerate}
\item When $\mathfrak p \nmid \mathfrak f$, $\deg_\mathfrak p \tau =1$, we have for $c \in C_\mathfrak f$
\begin{align*}
&\Phi_\tau(\Gamma(c;D,\mathfrak a_c)) \equiv \frac{\pi_{\mathfrak p}^{\frac{\zeta(0,c)}{h_F^+}}}
{\prod_{\tilde c \in \phi^{-1}( [\mathfrak p]c)}\exp_p(X_p(\tilde c;D,\mathfrak a_c))}
\Gamma([\mathfrak p]c;D,\mathfrak p \mathfrak a_c) \mod \mu_\infty.
\end{align*}
Here $[\mathfrak p] \in C_\mathfrak f$ denotes the ideal class of $\mathfrak p$, 
$\phi \colon C_{\mathfrak{fp}} \ra C_\mathfrak f$ the natural projection.
\item When $\mathfrak p \mid \mathfrak f$, we have for $c \in C_\mathfrak f$
\begin{align*}
\Phi_\tau(\Gamma(c)) \equiv \Gamma(c_\tau c) \mod \mu_\infty.
\end{align*}
Here $c_\tau \in C_\mathfrak f$ denotes the ideal class corresponding to $\tau|_{H_\mathfrak f}$ via the Artin map:
\begin{align*}
\tau \in W_{F_\mathfrak p} \subset \mathrm{Gal}(\overline{\mathbb Q_p}/F_\mathfrak p)
\subset \mathrm{Gal}(\oq/F) \st{|_{H_\mathfrak f}}\ra \mathrm{Gal}(H_\mathfrak f/F) \cong C_\mathfrak f.
\end{align*}
\end{enumerate}
\end{cnj}

\begin{prp} \label{cst} 
The truth of {\rm Conjecture \ref{maincnj}} does not depend on the choices of $D,\mathfrak a_c,\pi_\mathfrak q$.
\end{prp}

\begin{proof}
For (ii), the assertion is trivial since $\Gamma(c) \bmod \mu_\infty$ does not depend on any choice.
For (i), let $\tilde\tau \in \mathrm{Aut}(\mathbb C/F)$ be any lift 
of $\tau \in W_{F_\mathfrak p} \subset \mathrm{Gal}(\overline{F_\mathfrak p}/F_\mathfrak p) \subset \mathrm{Gal}(\overline{\mathbb Q}/F)$.
We need to show that the variations of 
\begin{align*}
\frac{\tilde\tau\left(\exp(X([\mathfrak p]c;D,\mathfrak p \mathfrak a_c))\right)\pi_{\mathfrak p}^{\frac{\zeta(0,c)}{h_F^+}}}{\exp(X(c;D,\mathfrak a_c))} ,\ 
\prod_{\tilde c \in \phi^{-1}( [\mathfrak p]c)}\exp_p(X_p(\tilde c;D,\mathfrak a_c))
\end{align*}
coincide $\bmod \mu_\infty$ when we replace $D,\mathfrak a_c,\pi_\mathfrak q$.
By (\ref{indep1}), the variation of $\exp(X([\mathfrak p]c;D,\mathfrak p \mathfrak a_c))$ is a rational power of an element in $E_+$, 
which is fixed under $\tilde \tau$ (up to $\mu_\infty$).
Therefore, when we consider the variation, we may drop $\tilde \tau$ of the former one.
Then it becomes $\prod_{\tilde c \in \phi^{-1}([\mathfrak p]c)}\exp(X(\tilde c;D,\mathfrak a_c))$ by Proposition \ref{575859}.
Hence the assertion follows from (\ref{indep2}).
\end{proof}

\begin{rmk}
{\rm Conjecture \ref{maincnj}-(i)} and {\rm Conjecture \ref{maincnj}-(ii)} are consistent with each other in the following sense:
Let $c \in C_{\mathfrak{fp}}$ with $\mathfrak p \nmid \mathfrak f$, $\phi \colon C_{\mathfrak{fp}} \ra C_\mathfrak f$ the natural projection.
We assume that $\mathfrak p \mid \mathfrak a_c$.
Then {\rm Proposition \ref{relofgc}-(iv)} states that 
\begin{align*}
\prod_{\tilde c \in \phi^{-1}(c)} \Gamma(\tilde c) \exp_p(X_p(\tilde c;D,\mathfrak p^{-1}\mathfrak a_c)) \equiv 
\frac{\pi_\mathfrak p^{\frac{\zeta(0,[\mathfrak p]^{-1}c)}{h_F^+}}
\Gamma(c;D,\mathfrak a_c)}{\Gamma([\mathfrak p]^{-1}c;D,\mathfrak p^{-1}\mathfrak a_c)} \mod \mu_\infty.
\end{align*}
{\rm Conjecture \ref{maincnj}-(ii)} implies, under $\Phi_\tau$ $(\tau \in W_{F_\mathfrak p}$, $\deg_\mathfrak p \tau =1)$, 
the left-hand side changes to   
\begin{align} \label{lhs}
\prod_{\tilde c \in \phi^{-1}(c)} \Gamma(c_\tau \tilde c)
\exp_p(X_p(\tilde c;D,\mathfrak p^{-1}\mathfrak a_c)) \bmod \mu_\infty.
\end{align}
{\rm Conjecture \ref{maincnj}-(i)} implies that the right-hand side changes to   
\begin{align*}
\frac{\pi_{\mathfrak p}^{\frac{\zeta(0,c)}{h_F^+}}
\prod_{\tilde c \in \phi^{-1}(c)}
\exp_p(X_p(\tilde c;D,\mathfrak p^{-1}\mathfrak a_c))}{\prod_{\tilde c \in \phi^{-1}([\mathfrak p]c)}
\exp_p(X_p(\tilde c;D,\mathfrak a_c))}\frac{\Gamma([\mathfrak p]c;D,\mathfrak p \mathfrak a_c)}{\Gamma(c;D,\mathfrak a_c)} \bmod \mu_\infty,
\end{align*}
which is, by {\rm Proposition \ref{relofgc}-(iv)} again, equal to
\begin{align} \label{rhs}
\prod_{\tilde c \in \phi^{-1}(c)}\exp_p(X_p(\tilde c;D,\mathfrak p^{-1}\mathfrak a_c))
\prod_{\tilde c \in \phi^{-1}([\mathfrak p]c)} \Gamma(\tilde c).
\end{align}
Noting that 
$\{c_\tau \tilde c \in C_\mathfrak {fp} \mid \phi(\tilde c)=c\}=\{\tilde c \in C_\mathfrak {fp} \mid \phi(\tilde c)=[\mathfrak p] c\}$,
we obtain {\rm (\ref{lhs}) $=$ (\ref{rhs})}.
\end{rmk}

\subsection{Relation to the reciprocity law on Stark's units}

In this subsection, we prove that Conjecture \ref{maincnj} is a refinement of the reciprocity law (\ref{rlofSu}) on Stark's units.
Let $C_\mathfrak f,H_\mathfrak f$ be the narrow ray class group, the narrow ray class field, 
modulo $\mathfrak f$ over a totally real field $F$, as in previous subsections.
Let $H$ be an intermediate field of $H_\mathfrak f/F$, $\phi_H\colon C_\mathfrak f \ra \mathrm{Gal}(H/F)$ the Artin map.
We denote by $\infty_\mathrm{id}$ the real place corresponding to $\mathrm{id}\colon F \hookrightarrow \mathbb R$.

\begin{thm} \label{cns}
{\rm Conjectures \ref{Yc}, \ref{maincnj}} imply {\rm (\ref{rl})}.
\end{thm}

\begin{proof}
Assume that $\infty_\mathrm{id}$ splits completely in $H/F$, i.e., $s_{\mathrm{id}} \in \ker \phi_H$.
Then we can write
\begin{align*}
\prod_{c \in \phi_H^{-1}(\sigma)}\exp(X(c))=\left(\prod_{c \in \phi_H^{-1}(\sigma)}\exp(X(c))\exp(X(s_{\mathrm{id}}c))\right)^{\frac{1}{2}}.
\end{align*}
By Proposition \ref{prpofp}-(ii), Conjecture \ref{Yc} implies the algebraicity of the right-hand side.
In the following of this proof, we vary a prime ideal $\mathfrak p$ of $F$.
Simultaneously, we vary a rational prime $p$ and an embedding $\mathrm{id} \colon F \hookrightarrow \mathbb C_p$ 
satisfying $\mathrm{id}$ induces $\mathfrak p$.
In particular we regard the Weil group $W_{F_\mathfrak p}$ as a subgroup of the decomposition group 
$D_\mathfrak p \subset \mathrm{Gal}(\oq/F)$ at $\mathfrak p$.
Then it suffices to show that (\ref{rl}) for $\tau \in W_{F_\mathfrak p}$ follows form Conjecture \ref{maincnj} for $\mathfrak p$.
First assume that $\mathfrak p \mid \mathfrak f$.
By Proposition \ref{prispr} we have
\begin{align*}
\prod_{c \in \phi_H^{-1}(\sigma)}\Gamma(c)\equiv \prod_{c \in \phi_H^{-1}(\sigma)}\frac{\exp(X(c))}{\exp_p(X_p(c))} \mod \mu_\infty.
\end{align*}
Then (\ref{rl}) follows from Conjecture \ref{maincnj}-(ii) since $\Phi_\tau$ is $\tau$-semilinear.
Next, assume that $\mathfrak p \nmid \mathfrak f$.
By Proposition \ref{prispr} again we have
\begin{align*}
\prod_{c \in \phi_H^{-1}(\sigma)}\Gamma(c;D,\mathfrak a_c)\equiv \prod_{c \in \phi_H^{-1}(\sigma)}\exp(X(c;D,\mathfrak a_c)) \mod \mu_\infty.
\end{align*}
Hence, when $\deg_\mathfrak p \tau=1$, Conjecture \ref{maincnj}-(i) implies that
\begin{align*}
&\tau(\prod_{c \in \phi_H^{-1}(\sigma)}\exp(X(c;D,\mathfrak a_c))) \\
&\equiv \frac{\pi_{\mathfrak p}^{\sum_{c \in \phi_H^{-1}(\sigma)}\frac{\zeta(0,c)}{h_F^+}}}
{\prod_{\tilde c \in \tilde\phi_H^{-1}((\frac{H/F}{\mathfrak p})\sigma)}\exp_p(X_p(\tilde c;D,\mathfrak a_c))}
\prod_{c \in \phi_H^{-1}(\sigma)}\exp(X([\mathfrak p]c;D,\mathfrak p\mathfrak a_c)) \mod \mu_\infty,
\end{align*}
where $\tilde \phi_H\colon C_{\mathfrak f\mathfrak p}\ra \mathrm{Gal}(H/F)$ denotes the Artin map.
Concerning $X_p(\tilde c;D,\mathfrak a_c)$, 
we take $\mathfrak a_c$ associated with the ideal class $c \in C_\mathfrak f$ satisfying that 
the image of $\tilde c \in C_\mathfrak{fp}$ under the natural projection is equal to $[\mathfrak p]c$.
It suffices to show that 
\begin{align*}
&\prod_{\tilde c \in \tilde\phi_H^{-1}((\frac{H/F}{\mathfrak p})\sigma)}\exp_p(X_p(\tilde c;D,\mathfrak a_c)) \\
&\equiv \pi_{\mathfrak p}^{\sum_{c \in \phi_H^{-1}(\sigma)}\frac{\zeta(0,c)}{h_F^+}} 
\frac{\prod_{c \in \phi_H^{-1}(\sigma)}\exp_q(X_q([\mathfrak p]c;D,\mathfrak p\mathfrak a_c))}
{\prod_{c \in \phi_H^{-1}(\sigma)}\exp_q(X_q(c;D,\mathfrak a_c))} \mod \mu_\infty
\end{align*}
for a prime ideal $\mathfrak q$ dividing $\mathfrak f$ lying above a rational prime $q$.
By Proposition \ref{575859}, the right-hand side is equal to
\begin{align*}
\prod_{\tilde c \in \tilde\phi_H^{-1}((\frac{H/F}{\mathfrak p})\sigma)}\exp_q(X_q(\tilde c;D,\mathfrak a_c)) \bmod \mu_\infty,
\end{align*}
which does not depend on the prime ideal $\mathfrak q$ dividing $\mathfrak{fp}$ by \cite[Theorem 4]{Ka4}.
Then the assertion is clear. 
\end{proof}

\subsection{Relation to refinements of Gross' conjecture in \cite{KY1}, \cite{KY2}}

Yoshida and the author formulated conjectures concerning $X_p(c)$ also in \cite{KY1}, \cite{KY2}.
In \cite{KY1}, we formulated a conjecture \cite[Conjecture A$'$]{KY1} equivalent to (\ref{rGc}), which is a refinement of Conjecture \ref{GSc} by Gross.
In \cite{KY2}, we rewrite and generalize this conjecture in terms of the symbol $Q^{(i)}$ defined by \cite[(4.2)]{KY2}:
roughly speaking, \cite{KY2} focuses on the absolute Frobenius action on $H_{\mathrm{cris}}(M(\chi))$,
whereas this paper focuses on the absolute Frobenius action on the $p$-adic period of $M(\chi)$.
(Moreover \cite[Conjecture Q]{KY2} is a statement $\bmod (K^\times)^\mathbb Q$, not $\bmod \mu_\infty$.)
These are consistent with each other under the following further refinement (\ref{assump}) of Conjecture \ref{Yc}.
In particular, as we shall see in Theorem \ref{cns2} below, Conjecture \ref{maincnj}-(i) and (\ref{assump}) imply (\ref{rGc}).

Let the notation be as in Conjecture \ref{Yc}.
We put 
\begin{align*}
\Xi:=\sum_{c' \in C_{\mathfrak f}} \tfrac{\zeta(0,c')}{[H_\mathfrak f:H_{\mathrm{CM}}]} \mathrm{Art}(c')^{-1} \in I_{H_{\mathrm{CM}},\mathbb Q}
\end{align*}
and take $\chi_{w\Xi}$ as in Proposition \ref{pphi}. By (\ref{pptoPP}) we can rewrite
\begin{align*}
\Gamma(c;D,\mathfrak a_c) 
\equiv \frac{\exp(X(c;D,\mathfrak a_c))}{\mathrm{Per}(\mathrm{Art}(c),\chi_{w\Xi})^{\frac{1}{2w}}}
\mathrm{Per}_p(\mathrm{Art}(c),\chi_{w\Xi})^{\frac{1}{2w}} \mod \mu_\infty \quad (\mathfrak p \nmid \mathfrak f).
\end{align*}
In this subsection, we assume that we can take $\chi_{w\Xi}$ satisfying 
\begin{align} \label{assump}
\exp(X(c;D,\mathfrak a_c))
\equiv 
\mathrm{Per}(\mathrm{Art}(c),\chi_{w\Xi})^{\frac{1}{2w}}
\mod (H_{\mathrm{CM}}^\times)^\mathbb Q.
\end{align}
When $F=\mathbb Q$, the assumption (\ref{assump}) also follows from Rohrlich's formula (\ref{rh}).

\begin{exm}
{\rm Example \ref{mawahc}} is a supporting evidence of not only {\rm Conjecture \ref{Yc}}, but also the assumption {\rm (\ref{assump})}:
Let $\chi$ be an algebraic Hecke character of $K=\mathbb Q(\sqrt{2\sqrt{5}-26})$ whose infinity type is $w(\mathrm{id}-\rho)$.
Let $F_0:=\mathbb Q(\sqrt{41})$, which is the maximal totally real subfield of the reflex field.
Then we see that 
\begin{align*}
M(\chi)\times_K KF_0 =(H^1(C\times_{F_0} KF_0) \otimes_K H^1(C'\times_{F_0} KF_0)\otimes_\mathbb Q \mathbb Q(1))^{\otimes w}
\end{align*}
as a motive defined over $KF_0$, with coefficients in $K$. 
We note that 
\begin{align*}
\omega_{\mathrm{id}} \otimes \omega_{\mathrm{id}}' \in H_{\mathrm{dR}}(M(\chi)\times_K KF_0)
\end{align*}
is fixed under the conjugation map of $F_0$,
that is, $\omega_{\mathrm{id}} \otimes \omega_{\mathrm{id}}' \in H_{\mathrm{dR}}(M(\chi))$, 
although each $\omega_{\mathrm{id}},\omega_{\mathrm{id}}'$ is defined only over $KF_0$.
Hence we have
\begin{align*}
\mathrm{Per}(\mathrm{id},\chi) \equiv \left((2\pi i)^{-1} \int \omega_{\mathrm{id}}\int \omega_{\mathrm{id}}'\right)^w \mod K^\times.
\end{align*}
Therefore the numerical example {\rm (\ref{num})} satisfies the assumption {\rm (\ref{assump})}.
\end{exm}

Under the assumption (\ref{assump}), 
the absolute Frobenius action on $H_{\mathrm{cris}}(M(\chi_{w\Xi}))$ coincides with that on the $p$-adic period of $M(\chi_{w\Xi})$:
let $\mathfrak P$ be the prime ideal of $H_{\mathrm{CM}}$ corresponding to the $p$-adic topology, 
$\pi_\mathfrak P$ a generator of $\mathfrak P^{h}$ with $h$ the class number of $H_{\mathrm{CM}}$.
Let $\tau \in W_{(H_{\mathrm{CM}})_\mathfrak P}$ with $\deg_\mathfrak P \tau=1$, $c \in C_\mathfrak f$ with $\mathfrak p\nmid \mathfrak f$.
Under the assumption (\ref{assump}), Proposition \ref{defP} implies  
\begin{align} 
&\Phi_\tau(\Gamma(c;D,\mathfrak a_c)) \notag \\
&\equiv \frac{\exp(X(c;D,\mathfrak a_c))}{\mathrm{Per}(\mathrm{Art}(c),\chi_{w\Xi})^{\frac{1}{2w}}}
\Phi_\tau\left(\mathrm{Per}_p(\mathrm{Art}(c),\chi_{w\Xi})^{\frac{1}{2w}}\right)
\equiv \alpha_c \Gamma(c;D,\mathfrak a_c), \label{assimply}
\end{align}
where we put
\begin{align*}
\alpha_c&:=\mathrm{Art}(c)((w\Xi-\rho\circ w\Xi)(\pi_\mathfrak P))^{\frac{1}{2wh}} 
\equiv \prod_{c' \in C_{\mathfrak f}} \mathrm{Art}(cc'^{-1})(\pi_\mathfrak P)^{\frac{\zeta(0,c')}{h[H_\mathfrak f:H_{\mathrm{CM}}]}} \mod \mu_\infty.
\end{align*}

\begin{thm} \label{cns2}
{\rm Conjecture \ref{maincnj}-(i)} and the assumption {\rm (\ref{assump})} imply {\rm (\ref{rGc})}.
\end{thm}

\begin{proof}
Without loss of generality we may assume that
\begin{itemize}
\item $\mathrm{ord}_\mathfrak p\mathfrak f=1$.
\item $H$ is the maximal subfield of $H_\mathfrak f$ where $\mathfrak p$ splits completely.
\end{itemize}
We may assume the former one since Proposition \ref{575859} implies that
\begin{align*}
\prod_{c \in \phi_H^{-1}(\sigma)}\frac{\exp(X(c))}{\exp_p(X_p(c))} \bmod \mu_\infty 
\end{align*}
does not change if we replace $\mathfrak f$ with $\mathfrak f\mathfrak p^{-1}$ whenever $\mathrm{ord}_\mathfrak p\mathfrak f \geq 2$.
The latter one is obvious from the form of (\ref{rGc}).
We use the following notation:
\begin{align*}
&\mathfrak f=\mathfrak f_0\mathfrak p \ (\mathfrak p \nmid \mathfrak f_0), \\
&\phi_H\colon C_\mathfrak f \ra \mathrm{Gal}(H/F), \\
&\phi_{H,0}\colon C_{\mathfrak f_0} \ra \mathrm{Gal}(H/F), \\
&\phi \colon C_\mathfrak f \ra C_{\mathfrak f_0}, \text{ which satisfy } \phi_{H,0}\circ \phi=\phi_H.
\end{align*}
We consider the maximal CM-subfield $H_{\mathrm{CM}}$ of the narrow ray class field $H_{\mathfrak f_0}$ modulo $\mathfrak f_0$.
Let $\tau \in W_{(H_{\mathrm{CM}})_\mathfrak P} $ satisfy $\deg_\mathfrak P \tau =1$.
Considering $W_{(H_{\mathrm{CM}})_\mathfrak P} \subset W_{F_\mathfrak p}$, we put $f:=\deg_\mathfrak p \tau$, 
which equals the order of $(\frac{H/F}{\mathfrak p}) \in \mathrm{Gal}(H/F)$.
By using Conjecture \ref{maincnj}-(i) $f$-times, we have
\begin{align*}
\Phi_\tau(\Gamma(c;D,\mathfrak a_c))
\equiv \frac{\pi_{\mathfrak p}^{\sum_{i=0}^{f-1}\frac{\zeta(0,[\mathfrak p]^i c)}{h_F^+}}\Gamma([\mathfrak p]^f c;D,\mathfrak p^f \mathfrak a_c)}
{\prod_{i=1}^{f}\prod_{\tilde c \in \phi^{-1}( [\mathfrak p]^ic)}
\exp_p(X_p(\tilde c;D,\mathfrak p^{i-1}\mathfrak a_c))} \mod \mu_\infty \quad (c\in C_{\mathfrak f_0}).
\end{align*}
By taking $\prod_{c \in \phi_{H,0}^{-1}(\sigma)}$, 
we obtain 
\begin{align*}
\Phi_\tau(\prod_{c \in \phi_{H,0}^{-1}(\sigma)} \Gamma(c;D,\mathfrak a_c)) 
\equiv \frac{\pi_{\mathfrak p}^{\sum_{c \in \phi_{H,0}^{-1}(\sigma)}\sum_{i=0}^{f-1}\frac{\zeta(0,[\mathfrak p]^i c)}{h_F^+}}
\prod_{c \in \phi_{H,0}^{-1}(\sigma)}\Gamma([\mathfrak p]^f c;D,\mathfrak p^f \mathfrak a_c)}
{\prod_{c \in \phi_{H,0}^{-1}(\sigma)}\prod_{i=1}^{f}
\prod_{\tilde c \in \phi^{-1}( [\mathfrak p]^ic)}\exp_p(X_p(\tilde c;D,\mathfrak p^{i-1}\mathfrak a_c))} 
\mod \mu_\infty.
\end{align*}
By (\ref{assimply}), under the assumption {\rm (\ref{assump})}, the left-hand side equals 
\begin{align*}
(\prod_{c \in \phi_{H,0}^{-1}(\sigma)} \prod_{c' \in C_{{\mathfrak f}_0}} 
\mathrm{Art}(cc'^{-1})(\pi_\mathfrak P)^{\frac{\zeta(0,c')}{h[H_{\mathfrak f_0}:H_{\mathrm{CM}}]}} )
\prod_{c \in \phi_{H,0}^{-1}(\sigma)}\Gamma(c;D,\mathfrak a_c) \mod \mu_\infty,
\end{align*}
where $\mathfrak f$ in (\ref{assimply}) is $\mathfrak f_0$ here.
By noting that $\mathfrak p$ is unramified in $H_{\mathrm{CM}}/F$ since $\mathfrak p \nmid \mathfrak f_0$
and that $H$ is the decomposition field of $\mathfrak p$, we have
\begin{align*}
N_{H_{\mathrm{CM}}/H}(\mathfrak P)=\mathfrak P_H^{[H_{\mathrm{CM}}:H]}.
\end{align*}
It follows that
\begin{align*}
\prod_{c \in \phi_{H,0}^{-1}(\sigma)} \prod_{c' \in C_{{\mathfrak f_0}}} 
\mathrm{Art}(cc'^{-1})(\pi_\mathfrak P)^{\frac{\zeta(0,c')}{h[H_{\mathfrak f_0}:H_{\mathrm{CM}}]}} 
=\sigma\left(\prod_{c' \in C_{\mathfrak f_0}} 
\mathrm{Art}(c'^{-1})(\pi_\mathfrak P)^{\frac{\zeta(0,c')|\ker \phi_{H,0}|}{h[H_{\mathfrak f_0}:H_{\mathrm{CM}}]}}\right)
\end{align*}
is equal to $(\alpha^\sigma)^{\frac{1}{h_H}}$ of (\ref{rGc}) since $|\ker \phi_{H,0}|=[H_{\mathfrak f_0}:H]$.
Hence it suffices to show that 
\begin{align*}
&\prod_{c \in \phi_{H,0}^{-1}(\sigma)}\prod_{i=1}^{f}
\prod_{\tilde c \in \phi^{-1}( [\mathfrak p]^ic)}\exp(X(\tilde c;D,\mathfrak p^{i-1}\mathfrak a_c)) \\
&\equiv 
\frac{\pi_{\mathfrak p}^{\sum_{c \in \phi_{H,0}^{-1}(\sigma)}\sum_{i=0}^{f-1}\frac{\zeta(0,[\mathfrak p]^i c)}{h_F^+}}
\prod_{c \in \phi_{H,0}^{-1}(\sigma)}\Gamma([\mathfrak p]^f c;D,\mathfrak p^f \mathfrak a_c)}
{\prod_{c \in \phi_{H,0}^{-1}(\sigma)} \Gamma(c;D,\mathfrak a_c)}
\mod \mu_\infty.
\end{align*}
On the right-hand side, the ``period part'' of $\prod_{c \in \phi_{H,0}^{-1}(\sigma)} \Gamma(c;D,\mathfrak a_c)$ is equal to that of 
$\prod_{c \in \phi_{H,0}^{-1}(\sigma)}\Gamma([\mathfrak p]^f c;D,\mathfrak p^f \mathfrak a_c)$ since $\phi_{H,0}([\mathfrak p]^f)=\mathrm{id}$.
Therefore the problem is reduced to showing that 
\begin{align*}
\prod_{i=1}^{f}
\prod_{\tilde c \in \phi^{-1}( [\mathfrak p]^ic)}\exp(X(\tilde c;D,\mathfrak p^{i-1}\mathfrak a_c)) 
\equiv 
\frac{\pi_{\mathfrak p}^{\sum_{i=0}^{f-1}\frac{\zeta(0,[\mathfrak p]^i c)}{h_F^+}}\exp(X([\mathfrak p]^f c;D,\mathfrak p^f \mathfrak a_c))}{\exp(X(c;D,\mathfrak a_c))}
\mod \mu_\infty
\end{align*}
for each $c \in \phi_{H,0}^{-1}(\sigma)$.
Since the right-hand side is equal to
\begin{align*}
\prod_{i=1}^{f}
\frac{\pi_{\mathfrak p}^{\frac{\zeta(0,[\mathfrak p]^{i-1} c)}{h_F^+}}\exp(X([\mathfrak p]^{i} c;D,\mathfrak p^{i} \mathfrak a_c))}
{\exp(X([\mathfrak p]^{i-1} c;D,\mathfrak p^{i-1} \mathfrak a_c))}
\mod \mu_\infty,
\end{align*}
the assertion follows from Proposition \ref{575859}.
\end{proof}

\begin{rmk}
{\rm Conjecture \ref{maincnj}-(ii)} is also consistent with {\rm (\ref{rGc})} in the following sense:
assume that $\mathfrak p \mid \mathfrak f$ and that $\mathfrak p$ splits completely in $H/F$ as in {\rm (\ref{rGc})}.
Then by {\rm Proposition \ref{prispr}} we have   
\begin{align*}
\prod_{c \in \phi_H^{-1}(\sigma)}\Gamma(c) 
\equiv \prod_{c \in \phi_H^{-1}(\sigma)}\frac{\exp(X(c))}{\exp_p(X_p(c))}  \mod \mu_\infty \quad (\sigma \in \mathrm{Gal}(H/F)).
\end{align*}
Hence {\rm Conjecture \ref{maincnj}-(ii)} implies 
\begin{align*}
\tau\left(\prod_{c \in \phi_H^{-1}(\sigma)}\frac{\exp(X(c))}{\exp_p(X_p(c))}\right)
\equiv \prod_{c \in \phi_H^{-1}(\tau|_H \sigma)}\frac{\exp(X(c))}{\exp_p(X_p(c))}  \mod \mu_\infty \quad (\tau \in W_{F_\mathfrak p}).
\end{align*}
Note that {\rm (\ref{rGc})} implies that it holds true for any $\tau \in \mathrm{Gal}(\oq/F)$,
not only for $\tau \in W_{F_\mathfrak p} \subset  \mathrm{Gal}(\oq/F)$.
\end{rmk}

\section{The case when $F=\mathbb Q$}

In this section, we prove the following theorem.
We note that  the assumption (\ref{assump}), which is a strong version of Conjecture \ref{Yc}, follows from Rohrlich's formula (\ref{rh}) in the case $F=\mathbb Q$.

\begin{thm} \label{cns3}
When $F=\mathbb Q$, $p\neq 2$, {\rm Conjecture \ref{maincnj}} holds true.
\end{thm}

This theorem was essentially proved in \cite{Ka2} by using Coleman's results \cite{Co} on the absolute Frobenius action on Fermat curves.
Let $F=\mathbb Q$, $\mathfrak f=(m)$ with $m \in \mathbb N$.
Then we see that 
\begin{align*}	
&C_{(m)}=\{[(r)] \mid 1\leq r \leq m,\ (r,m)=1\} \cong (\mathbb Z/m\mathbb Z)^\times, \\
&H_{(m)}=H_{\mathrm{CM}}=\mathbb Q(\zeta_m) \quad (\zeta_m:=e^{\frac{2\pi i}{m}}).
\end{align*}
We put for $r \in \mathbb N$ with $(r,m)=1$
\begin{align*}
c_{\frac{r}{m}}&:=[(r)] \in C_{(m)}, \\
\sigma_{\frac{r}{m}}&:=\mathrm{Art}(c_r)\colon \zeta_m \mt \zeta_m^r \in \mathrm{Gal}(\mathbb Q(\zeta_m)/\mathbb Q).
\end{align*}
When $(r,m)=d$, we regard
\begin{align*}
c_{\frac{r}{m}}:=c_{\frac{r/d}{m/d}}=[(r/d)] \in C_{(m/d)},\ \text{etc.}
\end{align*}
Then we have for $1\leq r \leq m$ with $(r,m)=d_r$,
\begin{align}
\zeta(0,c_{\frac{r}{m}})&=\tfrac{1}{2}-\tfrac{r}{m}, \notag \\
\exp(X(c_{\frac{r}{m}}))&=\exp(\zeta'(0,c_{\frac{r}{m}}))=\Gamma(\tfrac{r}{m})(m/d_r)^{\frac{r}{m}-\frac{1}{2}}(2\pi)^{-\frac{1}{2}}, \label{ex} \\
\exp_p(X_p(c_{\frac{r}{m}})) &\equiv \Gamma_p(\tfrac{r}{m})(m/d_r)_0^{\frac{r}{m}-\frac{1}{2}} \mod \mu_\infty \quad (p\mid \frac{m}{d_r}). \label{exp}
\end{align}
Here $m_0$ denotes the prime-to-$p$ part of an integer $m$, 
$\Gamma_p$ on $\mathbb Q_p-\mathbb Z_p$ is defined in \cite[Lemma 4.2]{Ka2}.
The first two lines follow from well-known formulas for the Hurwitz zeta function, noting that $\zeta(s,c_{\frac{r}{m}})=\sum_{k =0}^\infty (r/d_r+km/d_r)^{-s}$.
The $p$-adic counterpart follows from the definition of $X_p$, 
noting that we can write $\Gamma_p(z)=\exp_p(L\Gamma_{p,1}(z,(1)))$ with $L\Gamma_{p,1}$ in \cite[(5.10)]{Ka1}.
We note that $\exp(X(c_{\frac{r}{m}})),\exp_p(X_p(c_{\frac{r}{m}}))$ do not depend on the choices of $D,\mathfrak a_c$.

We consider differential forms $\eta_{\frac{r}{m},\frac{s}{m}}:=x^{r-1}y^{s-m} dx$ ($0< r,s < m$, $r+s\neq m$) of the $m$th Fermat curve $F_m\colon x^m+y^m=1$.
When $r+s<m$, $(r,s,r+s,m)=1$, we define
\begin{align*}
\Xi_{\frac{r}{m},\frac{s}{m}}
&:=\sum_{1\leq b \leq m,\ (b,m)=1,\ \langle \frac{br}{m}\rangle +\langle \frac{bs}{m}\rangle +\langle \frac{b(m-r-s)}{m}\rangle =1}\sigma_\frac{b}{m} \\
&=\sum_{1\leq b \leq m,\ (b,m)=1}(1-\varepsilon(\tfrac{br}{m},\tfrac{bs}{m}))\sigma_\frac{b}{m}.
\end{align*}
Here $\langle \alpha \rangle \in (0,1)$ denotes the fraction part of $\alpha \in \mathbb Q-\mathbb Z$
and we put for $\alpha,\beta \in \mathbb Q$ with $\alpha,\beta,\alpha+\beta \notin \mathbb Z$
\begin{align*}
\varepsilon(\alpha,\beta):=\langle \alpha \rangle + \langle \beta \rangle - \langle \alpha+\beta \rangle 
=
\begin{cases}
0 & (\langle \alpha \rangle+\langle \beta \rangle <1) \\
1 & (\langle \alpha \rangle+\langle \beta \rangle >1)
\end{cases}.
\end{align*}
Then $\Xi_{\frac{r}{m},\frac{s}{m}}$ is a CM-type of $\mathbb Q(\zeta_m)$.
Moreover, by \cite[Chap.\ III, \S 2.1]{Yo}, we have
\begin{align*}
[2\pi i: (2\pi i)_p][p_{\mathbb Q(\zeta_m)}:p_{\mathbb Q(\zeta_m),p}](\mathrm{id},\Xi_{\frac{r}{m},\frac{s}{m}})
=[\int_\gamma \eta_{\frac{r}{m} ,\frac{s}{m}}:\int_{\gamma,p} \eta_{\frac{r}{m} ,\frac{s}{m} }].
\end{align*}

\begin{prp} \label{PP}
Assume that $0< r,s < m$, $r+s\neq m$, $(r,s,r+s,m)=1$.
Then we have 
\begin{align*}
&[\mathrm{P}(c_{\frac{r}{m}})\mathrm{P}(c_{\frac{s}{m}})\mathrm{P}(c_{\frac{r+s}{m}})^{-1}:
\mathrm{P}_p(c_{\frac{r}{m}})\mathrm{P}_p(c_{\frac{s}{m}})\mathrm{P}_p(c_{\frac{r+s}{m}})^{-1}] \\
&\equiv 
[2\pi i: (2\pi i)_p]^{-\frac{1}{2}}[\int_\gamma \eta_{\frac{r}{m},\frac{s}{m}}:\int_{\gamma,p} \eta_{\frac{r}{m},\frac{s}{m}}] \mod \mu_\infty.
\end{align*}
\end{prp}

\begin{proof}
Assume that $(r,m)=d$. Then by Proposition \ref{prpofp}-(iii), (iv) we see that 
\begin{align*}
\mathrm{P}_*(c_{\frac{r}{m}})
&\equiv (2\pi i)_*^{\frac{1}{2}-\frac{r}{m}}p_{\mathbb Q(\zeta_{m/d}),*}
(\sigma_{\frac{r}{m}},\sum_{1\leq b \leq m/d,\ (b,m/d)=1}(\tfrac{1}{2}-\tfrac{b}{m/d}) \sigma_{\frac{b}{m/d}}) \\
&\equiv (2\pi i)_*^{\frac{1}{2}-\frac{r}{m}}p_{\mathbb Q(\zeta_{m/d}),*}
(\mathrm{id},\sum_{1\leq b \leq m/d,\ (b,m/d)=1}(\tfrac{1}{2}-\langle \tfrac{br/d}{m/d}\rangle) \sigma_{\frac{b}{m/d}}) \\
&\equiv (2\pi i)_*^{\frac{1}{2}-\frac{r}{m}}p_{\mathbb Q(\zeta_{m}),*}
(\mathrm{id},\sum_{1\leq b \leq m,\ (b,m)=1}(\tfrac{1}{2}-\langle \tfrac{br}{m}\rangle) \sigma_{\frac{b}{m}}).
\end{align*}
Therefore when $r+s<m$, by Proposition \ref{prpofp}-(ii), (iii), we can write 
\begin{align*}
&[\mathrm{P}(c_{\frac{r}{m}})\mathrm{P}(c_{\frac{s}{m}})\mathrm{P}(c_{\frac{r+s}{m}})^{-1}:
\mathrm{P}_p(c_{\frac{r}{m}})\mathrm{P}_p(c_{\frac{s}{m}})\mathrm{P}_p(c_{\frac{r+s}{m}})^{-1}] \\
&\equiv [2\pi i: (2\pi i)_p]^{\frac{1}{2}}[p_{\mathbb Q(\zeta_m)}:p_{\mathbb Q(\zeta_m),p}](\mathrm{id},\sum_{1\leq b \leq m,\ (b,m)=1}
(\tfrac{1}{2}-\varepsilon(\tfrac{br}{m},\tfrac{bs}{m}))\sigma_{\frac{b}{m}}) \\
&\equiv [2\pi i: (2\pi i)_p]^{\frac{1}{2}}
[p_{\mathbb Q(\zeta_m)}:p_{\mathbb Q(\zeta_m),p}](\mathrm{id},\Xi_{\frac{r}{m},\frac{s}{m}}-\tfrac{1}{2}\mathrm{Hom}(\mathbb Q(\zeta_m),\mathbb C)) \\
&\equiv [2\pi i: (2\pi i)_p]^{\frac{1}{2}}[p_{\mathbb Q(\zeta_m)}:p_{\mathbb Q(\zeta_m),p}](\mathrm{id},\Xi_{\frac{r}{m},\frac{s}{m}}) \\
&\equiv [2\pi i: (2\pi i)_p]^{-\frac{1}{2}}[\int_\gamma \eta_{\frac{r}{m},\frac{s}{m}}:\int_{\gamma,p} \eta_{\frac{r}{m},\frac{s}{m}}].
\end{align*}
When $r+s>m$, by Proposition \ref{prpofp}-(ii) and (\ref{2pii}), we have
\begin{align*}
&[\mathrm{P}(c_{\frac{r}{m}})\mathrm{P}(c_{\frac{s}{m}})\mathrm{P}(c_{\frac{r+s}{m}})^{-1}:
\mathrm{P}_p(c_{\frac{r}{m}})\mathrm{P}_p(c_{\frac{s}{m}})\mathrm{P}_p(c_{\frac{r+s}{m}})^{-1}] \\
&\equiv [\mathrm{P}(c_{\frac{m-r}{m}})\mathrm{P}(c_{\frac{m-s}{m}})\mathrm{P}(c_{\frac{2m-(r+s)}{m}})^{-1}:
\mathrm{P}_p(c_{\frac{m-r}{m}})\mathrm{P}_p(c_{\frac{m-s}{m}})\mathrm{P}_p(c_{\frac{2m-(r+s)}{m}})^{-1}]^{-1} \\
&\equiv [2\pi i: (2\pi i)_p]^{\frac{1}{2}}[\int_\gamma \eta_{\frac{m-r}{m},\frac{m-s}{m}}:\int_{\gamma,p} \eta_{\frac{m-r}{m},\frac{m-s}{m}}]^{-1} \\
&\equiv [2\pi i: (2\pi i)_p]^{-\frac{1}{2}} [\int_\gamma \eta_{\frac{r}{m},\frac{s}{m}}:\int_{\gamma,p} \eta_{\frac{r}{m},\frac{s}{m}}].
\end{align*}
Then the assertion is clear.
\end{proof}

By (\ref{ex}), (\ref{exp}) and Proposition \ref{PP}, we obtain the explicit relation 
between the period-ring-valued functions $\Gamma(c_{\frac{r}{m}}),\Gamma(c_{\frac{r}{m}};D,\mathfrak a_c)$ in this paper 
and the period-ring-valued beta function $\mathfrak B(\frac{i}{m},\frac{j}{m})$ defined in \cite[the first paragraph of \S 7]{Ka2}. 
First we assume that $p \mid m$, $p \nmid rs(r+s)$. We put $d_*:=(m,*)$ for $*=r,s,r+s$. 
Then we have
\begin{align*}
&\frac{\Gamma(c_{\frac{r}{m}})\Gamma(c_{\frac{s}{m}})}{\Gamma(c_{\frac{r+s}{m}})} \\
&\equiv
\frac{\frac{\Gamma(\frac{r}{m})(m/d_r)^{\frac{r}{m}-\frac{1}{2}}(2\pi)^{-\frac{1}{2}}\Gamma(\frac{s}{m})(m/d_s)^{\frac{s}{m}-\frac{1}{2}}(2\pi)^{-\frac{1}{2}}} 
{\Gamma(\langle\frac{r+s}{m}\rangle)(m/d_{r+s})^{\langle \frac{r+s}{m} \rangle -\frac{1}{2}}(2\pi)^{-\frac{1}{2}}}}
{\frac{\Gamma_p(\frac{r}{m})(m/d_r)_0^{\frac{r}{m}-\frac{1}{2}}\Gamma_p(\frac{s}{m})(m/d_s)_0^{\frac{s}{m}-\frac{1}{2}}} 
{\Gamma_p(\langle\frac{r+s}{m}\rangle)(m/d_{r+s})_0^{\langle \frac{r+s}{m} \rangle -\frac{1}{2}}}}
\left(\frac{(2\pi i)_p}{2\pi i}\right)^{-\frac{1}{2}}\frac{\int_{\gamma,p} \eta_{\frac{r}{m},\frac{s}{m}}}{\int_\gamma \eta_{\frac{r}{m},\frac{s}{m}}} \\
&\equiv(2\pi)^{-\frac{1}{2}} \left(p^{\mathrm{ord}_p m}\right)^{-\frac{1}{2}} \frac{B(\frac{r}{m},\frac{s}{m})}{B_p(\frac{r}{m},\frac{s}{m})}
\left(\frac{(2\pi i)_p}{2\pi i}\right)^{-\frac{1}{2}}\frac{\int_{\gamma,p} \eta_{\frac{r}{m},\frac{s}{m}}}{\int_\gamma \eta_{\frac{r}{m},\frac{s}{m}}} \\
&\equiv (2\pi i)_p^{-\frac{1}{2}} \left(p^{\mathrm{ord}_p m}\right)^{-\frac{1}{2}} \mathfrak B(\tfrac{r}{m},\tfrac{s}{m}) \mod \mu_\infty.
\end{align*}
Here we note that
\begin{align*}
&B_*(\tfrac{r}{m},\tfrac{s}{m}):=\frac{\Gamma_*(\frac{r}{m})\Gamma_*(\frac{s}{m})}{\Gamma_*(\frac{r+s}{m})} \quad (*=\emptyset,p), \\
&\Gamma(\langle \tfrac{r+s}{m} \rangle)  \langle \tfrac{r+s}{m} \rangle^{\varepsilon(\frac{r}{m},\frac{s}{m})}=\Gamma(\tfrac{r+s}{m}), \\ 
&\Gamma_p(\langle \tfrac{r+s}{m} \rangle)  \langle \tfrac{r+s}{m} \rangle^{\varepsilon(\frac{r}{m},\frac{s}{m})}
\left(p^{\mathrm{ord}_p m}\right)^{\varepsilon(\frac{r}{m},\frac{s}{m})} \equiv \Gamma(\tfrac{r+s}{m}) \mod \mu_\infty.
\end{align*}
The last line follows from the characterization \cite[(4.4)]{Ka2} of $\Gamma_p$.
Next, assume that $p \nmid m$. We drop the symbols $D,\mathfrak a_c$ from $\Gamma(c_{\frac{r}{m}};D,\mathfrak a_c)$. 
By definition we have
\begin{align*}
\frac{\Gamma(c_{\frac{r}{m}})\Gamma(c_{\frac{s}{m}})}{\Gamma(c_{\frac{r+s}{m}})} 
&\equiv\frac{\Gamma(\tfrac{r}{m})(m/d_r)^{\frac{r}{m}-\frac{1}{2}}(2\pi)^{-\frac{1}{2}}\Gamma(\tfrac{s}{m})(m/d_s)^{\frac{s}{m}-\frac{1}{2}}(2\pi)^{-\frac{1}{2}}} 
{\Gamma(\langle\tfrac{r+s}{m}\rangle)(m/d_{r+s})^{\langle \frac{r+s}{m} \rangle -\frac{1}{2}}(2\pi)^{-\frac{1}{2}}} 
\left(\frac{(2\pi i)_p}{2\pi i}\right)^{-\frac{1}{2}}\frac{\int_{\gamma,p} \eta_{\frac{r}{m},\frac{s}{m}}}{\int_\gamma \eta_{\frac{r}{m},\frac{s}{m}}} \\
&\equiv(2\pi)^{-\frac{1}{2}}
\frac{(m/d_r)^{\frac{r}{m}-\frac{1}{2}}(m/d_s)^{\frac{s}{m}-\frac{1}{2}}} {(m/d_{r+s})^{\langle \frac{r+s}{m} \rangle -\frac{1}{2}}}
B(\tfrac{r}{m},\tfrac{r}{m})\langle \tfrac{r+s}{m} \rangle^{\varepsilon(\frac{r}{m},\frac{s}{m})} 
\left(\frac{(2\pi i)_p}{2\pi i}\right)^{-\frac{1}{2}}\frac{\int_{\gamma,p} \eta_{\frac{r}{m},\frac{s}{m}}}{\int_\gamma \eta_{\frac{r}{m},\frac{s}{m}}} \\
&\equiv (2\pi i)_p^{-\frac{1}{2}}
\frac{(m/d_r)^{\frac{r}{m}-\frac{1}{2}}(m/d_s)^{\frac{s}{m}-\frac{1}{2}}} {(m/d_{r+s})^{\langle \frac{r+s}{m} \rangle -\frac{1}{2}}}
\mathfrak B(\tfrac{r}{m},\tfrac{r}{m}) \mod \mu_\infty.
\end{align*}
Hence \cite[Theorem 7.2]{Ka2} states that we have for $\tau \in W_{\mathbb Q_p}$ 
\begin{align}
&\Phi_\tau\left(\frac{\Gamma(c_{\frac{r}{m}})\Gamma(c_{\frac{s}{m}})}{\Gamma(c_{\frac{r+s}{m}})}\right) 
\equiv 
\frac{\Gamma(c_\tau c_{\frac{r}{m}})\Gamma(c_\tau c_{\frac{s}{m}})}{\Gamma(c_\tau c_{\frac{r+s}{m}})} \mod \mu_\infty 
\quad (p\mid m,\ p\nmid rs(r+s),\ p\neq 2), \label{thm1} \\
&\Phi_\tau\left( \frac{\Gamma(c_{\frac{r}{m}})\Gamma(c_{\frac{s}{m}})}{\Gamma(c_{\frac{r+s}{m}})} \right) 
\equiv 
\frac{\Gamma(c_{\frac{pr}{m}})\Gamma(c_{\frac{ps}{m}})}{\Gamma(c_{\frac{p(r+s)}{m}})} 
\frac{(m/d_r)^{\frac{r}{m}-\langle \frac{pr}{m} \rangle }(m/d_s)^{\frac{s}{m}-\langle \frac{ps}{m} \rangle}}
{(m/d_{r+s})^{\langle \frac{r+s}{m} \rangle -\langle \frac{p(r+s)}{m} \rangle }}  
\frac{p^{\frac{1}{2}-\varepsilon(\frac{r}{m},\frac{s}{m})}}
{\frac{\Gamma_p(\langle\frac{pr}{m}\rangle)\Gamma_p(\langle\frac{ps}{m}\rangle)}{\Gamma_p(\langle\frac{p(r+s)}{m}\rangle)}} \notag \\
&\hspace{200pt} \mod \mu_\infty \quad (p\nmid m,\ \deg_{(p)} \tau =1). \notag
\end{align}
When $p\nmid m$, by \cite[Remark 4.5]{Ka2}, we have 
\begin{align*}
\prod_{\tilde c\in C_{(pm/d_r)},\ \phi(\tilde c)=c_{\frac{pr}{m}}}\exp_p(X_p(\tilde c))
&\equiv \prod_{\substack{1\leq a \leq pm/d_r \\ (a,pm/d_r)=1 \\ a\equiv pr/d_r \bmod m/d_r}}\Gamma_p(\tfrac{a}{pm/d_r})(m/d_r)^{\frac{a}{mp}-\frac{1}{2}} \\
&\equiv \Gamma_p(\langle\tfrac{pr}{m}\rangle) (m/d_r)^{\langle\frac{pr}{m}\rangle-\frac{r}{m}} \mod \mu_\infty,
\end{align*}
where $\phi \colon C_{(pm/d_r)} \ra C_{(m/d_r)}$ denotes the natural projection. 
We put 
\begin{align*}
\gamma_p(c_{\frac{pr}{m}}):=p^{\zeta(0,c_{\frac{r}{m}})}\prod_{\tilde c\in C_{(pm/d_r)},\ \phi(\tilde c)=c_{\frac{pr}{m}}}\exp_p(X_p(\tilde c)).
\end{align*}
Then the case $p\nmid m$ can be rewritten into 
\begin{align} 
\Phi_\tau\left( \frac{\Gamma(c_{\frac{r}{m}})\Gamma(c_{\frac{s}{m}})}{\Gamma(c_{\frac{r+s}{m}})} \right) 
\equiv 
\frac{\Gamma(c_{\frac{pr}{m}})\Gamma(c_{\frac{ps}{m}})}{\Gamma(c_{\frac{p(r+s)}{m}})} 
\frac{\gamma_p(c_{\frac{p(r+s)}{m}})}{\gamma_p(c_{\frac{pr}{m}})\gamma_p(c_{\frac{ps}{m}})} \mod \mu_\infty \quad (p\nmid m,\ \deg \tau =1). \label{thm2} 
\end{align}
Now we derive the desired formula 
\begin{align} \label{thm3}
\Phi_\tau (\Gamma(c_{\frac{r}{m}})) \equiv 
\begin{cases}
\Gamma(c_\tau c_{\frac{r}{m}}) & (p\mid m,\ p\nmid r) \\
\Gamma(c_{\frac{pr}{m}})\gamma_p(c_{\frac{pr}{m}}) & (p\nmid m,\ \deg \tau =1)
\end{cases}
\mod \mu_\infty
\end{align} 
when $p\neq 2$, $(r,m)=1$, by induction on $t:=\mathrm{ord}_2\, m$.
When $t=0$, let $f$ be the order of $2$ in $(\mathbb Z/m\mathbb Z)^\times$.
Then we have
\begin{align*}
\prod_{k=0}^{f-1}\left(\frac{\Gamma(c_{\frac{2^kr}{m}})\Gamma(c_{\frac{2^kr}{m}})}{\Gamma(c_{\frac{2^{k+1}r}{m}})}\right)^{2^{f-1-k}}
=\frac{\Gamma(c_{\frac{r}{m}})^{2^f}}{\Gamma(c_{\frac{2^f r}{m}})}
=\Gamma(c_{\frac{r}{m}})^{2^f-1}.
\end{align*}
A similar property holds for $\gamma_p(c_{\frac{pr}{m}})$, so (\ref{thm3}) with $t=0$ follows from (\ref{thm1}), (\ref{thm2}).
The case $\mathrm{ord}_2\, m=t>0$ follows from the case $\mathrm{ord}_2\, m/2=t-1$ and (\ref{thm1}), (\ref{thm2}) since we have
\begin{align*}
\Gamma(c_{\frac{r}{m}})^2=\Gamma(c_{\frac{r}{m/2}})\frac{\Gamma(c_{\frac{r}{m}})\Gamma(c_{\frac{r}{m}})}{\Gamma(c_{\frac{r}{m/2}})}, \ 
\gamma_p(c_{\frac{pr}{m}})^2=\gamma_p(c_{\frac{pr}{m/2}})\frac{\gamma_p(c_{\frac{pr}{m}})\gamma_p(c_{\frac{pr}{m}})}{\gamma_p(c_{\frac{pr}{m/2}})}.
\end{align*}
Then the assertion of Theorem \ref{cns3} is clear.

\section{Partial results when $F\neq \mathbb Q$}

In this section, we prove the following partial result on Conjectures \ref{Yc}, \ref{maincnj}.

\begin{thm} \label{partial}
Let notations be as in {\rm Definition \ref{maindfn}, Conjecture \ref{maincnj}}, and the assumption {\rm (\ref{assump})}.
Let $K$ be an intermediate field of $H_\mathfrak f/F$, $\phi_K \colon C_\mathfrak f \ra \mathrm{Gal}(K/F)$ the Artin map, and $\sigma \in \mathrm{Gal}(K/F)$.
Assume that 
\begin{itemize}
\item There exists $f \in \mathbb N$ satisfying $\mathfrak f=(f)$.
\item $K$ is abelian over $\mathbb Q$. 
\end{itemize}
Then under some additional assumptions, our conjectures holds true if we add $\prod_{c \in \phi_K^{-1}(\sigma)}$ to both sides, as follows.
\begin{enumerate}
\item Concerning {\rm Conjecture \ref{Yc}}, we have
\begin{align*}
\prod_{c \in \phi_K^{-1}(\sigma)}\exp(X(c))
\equiv \prod_{c \in \phi_K^{-1}(\sigma)}
\left(\pi^{\zeta(0,c)}p_{H_{\mathrm{CM}}}(\mathrm{Art}(c),\sum_{c' \in C_{\mathfrak f}} \tfrac{\zeta(0,c')}{[H_\mathfrak f:H_{\mathrm{CM}}]} \mathrm{Art}(c'))\right)
\mod \overline{\mathbb Q}^\times.
\end{align*}
The same is true for the assumption {\rm (\ref{assump})}: for a suitable $\chi_{w\Xi}$, we have 
\begin{align*}
\prod_{c \in \phi_K^{-1}(\sigma)} \exp(X(c;D,\mathfrak a_c)) \equiv 
\prod_{c \in \phi_K^{-1}(\sigma)} \mathrm{Per}(\mathrm{Art}(c),\chi_{w\Xi})^{\frac{1}{2w}}
\mod (H_{\mathrm{CM}}^\times)^\mathbb Q.
\end{align*}
\item Concerning {\rm Conjecture \ref{maincnj}-(i)}, we additionally assume that $p\neq 2$ and that $p$ remains prime in $F$. 
Then we have for $\tau \in W_{F_\mathfrak p}$ with $\deg_\mathfrak p \tau =1$
\begin{align*}
&\Phi_\tau\left(\prod_{c \in \phi_K^{-1}(\sigma)}  \Gamma(c;D,\mathfrak a_c) \right) \\
&\equiv \prod_{c \in \phi_K^{-1}(\sigma)} \left(\frac{\pi_{\mathfrak p}^{\frac{\zeta(0,c)}{h_F^+}}}
{\prod_{\tilde c \in \phi^{-1}( [\mathfrak p]c)}\exp_p(X_p(\tilde c;D,\mathfrak a_c))}
\Gamma([\mathfrak p]c;D,\mathfrak p \mathfrak a_c)\right) \mod \mu_\infty.
\end{align*}
\item Concerning {\rm Conjecture \ref{maincnj}-(ii)},  we additionally assume that $p\neq 2$. Then we have for $\tau \in W_{F_\mathfrak p}$
\begin{align*}
\Phi_\tau\left(\prod_{c \in \phi_K^{-1}(\sigma)} \Gamma(c)\right) \equiv \prod_{c \in \phi_K^{-1}(\sigma)} \Gamma(c_\tau c)
\mod \mu_\infty.
\end{align*}
\end{enumerate}
In particular, if $H_\mathfrak f/\mathbb Q$ is abelian, then {\rm Conjecture \ref{Yc}} and the assumption {\rm (\ref{assump})} hold true.
Additionally assume that $p\neq 2$ and that $\mathfrak p=p\mathcal O_F \nmid \mathfrak f$ or $\mathfrak p\mid \mathfrak f$.
Then {\rm Conjecture \ref{maincnj}} also holds true.
\end{thm}

\begin{proof}
Assume that $\mathfrak f=(f)$ with $f \in \mathbb N$ and that an intermediate field $K$ of $H_\mathfrak f/F$ is abelian over $\mathbb Q$. 
Then $F$ is also abelian over $\mathbb Q$.
We put $G:=\mathrm{Gal}(K/\mathbb Q)$, $H:=\mathrm{Gal}(K/F)$ and denote by $\hat G,\hat H$ the groups of all characters of $G,H$ respectively.
For $\chi \in \hat H$, we consider the prime-to-$\mathfrak f$ part of the $L$-function:
\begin{align*}
L_\mathfrak f(s,\chi):=\sum_{c \in C_\mathfrak f} \chi(c)\zeta(s,c).
\end{align*}
Strictly speaking, we $\chi(c):=\chi(\phi_K(c))$.
Let $d$ be the least common multiple of $f$ and the conductor of $K/\mathbb Q$, 
$\zeta_d:=e^{\frac{2\pi i}{d}}$, $\mathrm{Art} \colon C_{(d)}\ra \mathrm{Gal}(\mathbb Q(\zeta_d)/\mathbb Q)$ the Artin map.
For $\psi \in \hat G$ we similarly consider 
\begin{align*}
L_d (s,\psi):=\sum_{c \in C_{(d)}} \psi(c)\zeta(s,c).
\end{align*}
Then we have
\begin{align} \label{funeq}
L_\mathfrak f(s,\chi)=\prod_{\psi \in \hat G,\ \psi|_H=\chi}L_d (s,\psi) \qquad (\chi \in \hat H).
\end{align}
It follows that 
\begin{align}
|H| \sum_{c \in \phi_K^{-1}(\sigma)}\zeta'(0,c) &= \sum_{\chi \in \hat H} \chi(\sigma^{-1})L_\mathfrak f'(0,\chi) \notag \\
&=\sum_{\chi \in \hat H} \chi(\sigma^{-1})\sum_{\psi_1 \in \hat G,\ \psi_1|_H=\chi}L_d'(0,\psi_1) 
\prod_{\psi_2 \in \hat G,\ \psi_2|_H=\chi,\ \psi_2 \neq \psi_1}L_d(0,\psi_1) \notag \\
&=\sum_{\psi_1 \in \hat G}\psi_1(\sigma^{-1}) L_d'(0,\psi_1) 
\prod_{\psi_2 \in \hat G,\ \psi_2|_H=\psi_1|_H,\ \psi_2 \neq \psi_1}L_d(0,\psi_2) \notag \\
&=\sum_{c \in C_{(d)}} r(c,\sigma) \cdot \zeta'(0,c), \label{zdtoz'}
\end{align}
where we put
\begin{align*}
r(c,\sigma):=\sum_{\psi_1 \in \hat G}\psi_1(\sigma^{-1})\psi_1(c)  \prod_{\psi_2 \in \hat G,\ \psi_2|_H=\psi_1|_H,\ \psi_2 \neq \psi_1}L_d(0,\psi_2) \in \mathbb Q.
\end{align*}
Recall that Conjecture \ref{Yc} in the case $F=\mathbb Q$ follows from Rohrlich's formula (\ref{rh}).
Then we have
\begin{align}
&\prod_{c \in \phi_K^{-1}(\sigma)} \exp(\zeta'(0,c))^{|H|} \notag \\
&\equiv \prod_{c\in C_{(d)}}\left((2\pi i)^{\zeta(0,c)}
p_{\mathbb Q(\zeta_d)}(\mathrm{Art}(c),
\sum_{c' \in C_{(d)}} \zeta(0,c') \mathrm{Art}(c')) \right)^{r(c,\sigma)} \mod \oq^\times. \label{prodofp}
\end{align}
The sum of the exponents of $2\pi i$ in (\ref{prodofp}) is equal to
\begin{align}
\sum_{c\in C_{(d)}}r(c,\sigma)\zeta(0,c)&=
\sum_{\psi_1 \in \hat G}\psi_1(\sigma^{-1})  L_d(0,\psi_1) \prod_{\psi_2 \in \hat G,\ \psi_2|_H=\psi_1|_H,\ \psi_2 \neq \psi_1}L_d(0,\psi_2) \notag \\
&=|G|\sum_{c \in \phi_K^{-1}(\sigma)}\zeta(0,c). \label{ztoz}
\end{align}
By Proposition \ref{prpofp}-(iii), we have
\begin{align} 
&\prod_{c\in C_{(d)}}p_{\mathbb Q(\zeta_d)}(\mathrm{Art}(c),\sum_{c' \in C_{(d)}} \zeta(0,c') \mathrm{Art}(c'))^{r(c,\sigma)} \notag \\
&\equiv p_{\mathbb Q(\zeta_d)}(\mathrm{id},\sum_{c\in C_{(d)}} \sum_{c' \in C_{(d)}} r(c,\sigma)\cdot \zeta(0,c') \mathrm{Art}(c^{-1}c')). \label{p1}
\end{align}
We easily see that
\begin{align} \label{p2}
\sum_{c\in C_{(d)}} \sum_{c' \in C_{(d)}} r(c,\sigma)\cdot \zeta(0,c') \mathrm{Art}(c^{-1}c')
&=|G| \cdot \mathrm{Inf} \left( \sum_{c' \in C_\mathfrak f}\zeta(0,c') \phi_K(c') \sigma^{-1}\right),
\end{align}
where $\mathrm{Inf} \colon I_{K,\mathbb Q} \ra I_{\mathbb Q(\zeta_d),\mathbb Q}$ denotes the inflation map.
Combining the above, we obtain 
\begin{align*}
\prod_{c \in \phi_K^{-1}(\sigma)} \exp(\zeta'(0,c))^{\frac{1}{[F:\mathbb Q]}} 
&\equiv \prod_{c \in \phi_K^{-1}(\sigma)} \left((2\pi i)^{\zeta(0,c)}
p_{\mathbb Q(\zeta_d)}(\mathrm{id},\mathrm{Inf} ( \sum_{c' \in C_\mathfrak f}\frac{\zeta(0,c')}{|\ker \phi_K|} \phi_K(c')\sigma^{-1})) \right) \\
&\equiv \prod_{c \in \phi_K^{-1}(\sigma)} P(c) \mod \oq^\times. 
\end{align*}
Here the last equality follows from Propositions \ref{prpofp}-(iv), 4-(iii).
Shintani's formula \cite[Theorem 1]{Ka3} states that 
\begin{align*}
\exp(\zeta'(0,c))=\prod_{\iota \in \HFR}\exp(X(\iota(c);\iota(D),\iota(\mathfrak a_c))),
\end{align*}
where $\iota(c),\iota(D),\iota(\mathfrak a_c)$ are an ideal class modulo $\iota(\mathfrak f)$, a Shintani domain of $\iota(F)$, 
an integral ideal of $\iota(F)$, respectively.
Since $\mathfrak f=(f)$ with $f \in \mathbb N$ and $F/\mathbb Q$ is normal,
we have $C_\mathfrak f=C_{\iota(\mathfrak f)}$ for any $\iota \in \HFR$.
Moreover the subset 
\begin{align*}
\{\iota(c)  \mid c \in \phi_K^{-1}(\sigma)\}
\end{align*}
of $C_\mathfrak f$ does not depend on $\iota \in \HFR$ since $K/\mathbb Q$ is abelian.
It follows that 
\begin{align*}
\prod_{\iota \in \HFR}\exp(X(\iota(c);\iota(D),\iota(\mathfrak a_c))) \equiv \exp(X(c;D,\mathfrak a_c))^{[F:\mathbb Q]} \mod E_+^\mathbb Q
\end{align*}
by (\ref{indep1}). 
Then the first assertion of (i) is clear.
Almost the same argument works for the assumption (\ref{assump}), although we need to decrease the ambiguity:
we may replace $\bmod \overline{\mathbb Q}^\times$ with $\bmod \mu_\infty$ 
since there exists a clothed path $\gamma$ ($:=n\gamma_n$ with $\gamma_n$ in \cite[Proposition 4.9]{Ot}) of $F_n(\mathbb C)$ satisfying 
$\int_\gamma \eta_{r,s} = \frac{\Gamma(\frac{r}{n})\Gamma(\frac{s}{n})}{\Gamma(\frac{r+s}{n})}$.

For (iii), the assumption $\mathfrak p \mid \mathfrak f=(f)$ implies that $p \mid d$ and that $\mathfrak p' \mid \mathfrak f$ for any $\mathfrak p' \mid p$.
Hence there exist $p$-adic interpolation functions $\zeta_p(s,c)$ of $\zeta(s,c)$ for $c \in C_\mathfrak f$ or $c \in C_{(d)}$.
We put 
\begin{align*}
L_{\mathfrak f,p}(s,\chi)&:=\sum_{c \in C_\mathfrak f} \chi(c)\zeta_p(s,c)  \qquad (\chi \in \hat G), \\
L_{d,p}(s,\psi)&:=\sum_{c \in C_{(d)}} \psi(c)\zeta_p(s,c) \qquad (\psi \in \hat H),
\end{align*}
which satisfy the same functional equation as (\ref{funeq}).
By Shintani's formula, its $p$-adic analogue \cite[Theorem 6.2]{Ka1}, \cite[Theorem 3.1]{KY1} 
and (\ref{indep2}), we have for $c \in C_\mathfrak f$
\begin{align*}
[\exp(\zeta'(0,c)):\exp_p(\zeta_p'(0,c))] \equiv 
[\exp(X(c;D,\mathfrak a_c))^{[F:\mathbb Q]}:\exp_p(X_p(c;D,\mathfrak a_c))^{[F:\mathbb Q]}] \mod \mu_\infty.
\end{align*}
Hence, by a similar argument to that in the previous paragraph, we obtain
\begin{align*}
\prod_{c \in \phi_K^{-1}(\sigma)} \Gamma(c) \equiv\prod_{c \in C_{(d)}} \Gamma(c)^{\frac{r(c,\sigma)}{|G|}}  \mod \mu_\infty.
\end{align*}
Then the assertion of (iii) is reduced to Theorem \ref{cns3}.

Concerning (ii), we consider
\begin{align*}
\gamma(\sigma)&:=\frac{\Phi_\tau\left(\prod_{c \in \phi_K^{-1}(\sigma)}  \Gamma(c;D,\mathfrak a_c) \right)}
{\prod_{c \in \phi_K^{-1}(\sigma)} \left(\frac{\pi_{\mathfrak p}^{\frac{\zeta(0,c)}{h_F^+}}}
{\prod_{\tilde c \in \phi^{-1}( [\mathfrak p]c)}\exp_p(X_p(\tilde c;D,\mathfrak a_c))}
\Gamma([\mathfrak p]c;D,\mathfrak p \mathfrak a_c)\right)} \quad (\sigma \in \mathrm{Gal}(K/F)), \\
\gamma(c)&:=\frac{\Phi_\tau \left(\Gamma(c) \right)}
{\prod_{i=1}^{[F:\mathbb Q]}\left(\frac{p^{\zeta(0,[(p)^{i-1}]c)}}{\prod_{\tilde c \in \phi^{-1}([(p)^i]c)}\exp_p(\zeta'(0,\tilde c))}\right)\Gamma([(p)^{[F:\mathbb Q]}]c)} \quad (c \in C_{(d)})
\end{align*}
for $\tau \in W_{F_\mathfrak p}$ with $\deg_\mathfrak p\tau=1$, 
where two canonical maps $C_{\mathfrak {fp}} \ra C_\mathfrak f$, $C_{(dp)} \ra C_{(d)}$ are denoted by the same symbol $\phi$.
We dropped the symbols $D,\mathfrak a_c$ in the case $F=\mathbb Q$.
Note that Theorem \ref{cns3} implies that 
\begin{align} \label{irt}
\gamma(c)\equiv 1 \bmod \mu_\infty
\end{align}
since $\deg_{(p)}\tau=[F:\mathbb Q]$ as an element of $W_{\mathbb Q_p}$ by the assumption $\mathfrak p=p\mathcal O_F$.
Let $\tilde\tau \in \mathrm{Aut}(\mathbb C/F)$ be any lift 
of $\tau \in W_{F_\mathfrak p} \subset \mathrm{Gal}(\overline{F_\mathfrak p}/F_\mathfrak p) \subset \mathrm{Gal}(\overline{\mathbb Q}/F)$.
Then the ratio
\begin{align*}
[\frac{\tilde\tau\left(\prod_{c \in \phi_K^{-1}(\sigma)}  \exp(X(c;D,\mathfrak a_c))\right)}{\prod_{c \in \phi_K^{-1}(\sigma)}  \exp(X([\mathfrak p]c;D,\mathfrak p\mathfrak a_c))}:
\prod_{c \in \phi_K^{-1}(\sigma)} 
\frac{\pi_{\mathfrak p}^{\frac{\zeta(0,c)}{h_F^+}}}{\prod_{\tilde c \in \phi^{-1}( [\mathfrak p]c)}\exp_p(X_p(\tilde c;D,\mathfrak a_c))}]
\end{align*}
does not depend on the choices of $D,\mathfrak a_c,\pi_\mathfrak q$ as we saw in the proof of Proposition \ref{cst}.
Therefore, by classical or $p$-adic Shintani's formula, the $[F:\mathbb Q]$th power of this ratio is equal to
\begin{align*}
[\frac{\tilde\tau\left(\prod_{c \in \phi_K^{-1}(\sigma)}  \exp(\zeta'(0,c))\right)}{\prod_{c \in \phi_K^{-1}(\sigma)}  \exp(\zeta'(0,[\mathfrak p]c))}:
\prod_{c \in \phi_K^{-1}(\sigma)} 
\frac{\pi_{\mathfrak p}^{\frac{[F:\mathbb Q]\zeta(0,c)}{h_F^+}}}
{\prod_{\tilde c \in \phi^{-1}( [\mathfrak p]c)}\exp_p(\zeta_p'(0,\tilde c))}] 
\end{align*}
$\mod \mu_\infty$.
By (\ref{zdtoz'}), (\ref{ztoz}) and $\mathrm{Frob}_\mathfrak p=\mathrm{Frob}_{p}^{[F:\mathbb Q]}$, we have
\begin{align*}
&\prod_{c \in \phi_K^{-1}(\sigma)}  \exp(\zeta'(0,c))=\prod_{c \in C_{(d)}}\exp(\zeta'(0,c))^\frac{r(c,\sigma)}{|H|}, \\
&\prod_{c \in \phi_K^{-1}(\sigma)}  \exp(\zeta'(0,[\mathfrak p]c))=\prod_{c \in C_{(d)}}\exp(\zeta'(0,c))^\frac{r(c,\mathrm{Frob}_\mathfrak p\sigma)}{|H|}
=\prod_{c \in C_{(d)}}\exp(\zeta'(0,[(p)]^{[F:\mathbb Q]}c))^\frac{r(c,\sigma)}{|H|}, \\
&\prod_{c \in \phi_K^{-1}(\sigma)} \pi_{\mathfrak p}^{\frac{[F:\mathbb Q]\zeta(0,c)}{h_F^+}}\equiv \prod_{c \in C_{(d)}} \left(p^{\zeta(0,c)}\right)^\frac{r(c,\sigma)}{|H|}.
\end{align*}
For $p$-adic invariants, we consider the $p$-adic interpolation of the prime-to-$p$ part of (\ref{funeq}):
\begin{align*}
\sum_{c\in C_\mathfrak f}\chi(c)\prod_{\tilde c \in \phi^{-1}(c)}\zeta_p(s,\tilde c)
=\prod_{\psi \in \hat G,\ \psi|_H=\chi}\sum_{c \in C_{(d)}}\psi(c)\prod_{\tilde c \in \phi^{-1}(c)}\zeta_p(s,\tilde c).
\end{align*}
Then we have
\begin{align*}
\prod_{c \in \phi_K^{-1}(\sigma)} \prod_{\tilde c \in \phi^{-1}( [\mathfrak p]c)}\exp_p(\zeta_p'(0,\tilde c))
\equiv 
\prod_{c \in C_{(d)}}\prod_{\tilde c \in \phi^{-1}(c)}\exp_p(\zeta_p'(0,\tilde c))^\frac{r'(c,\mathrm{Frob}_\mathfrak p\sigma)}{|H|} \mod \mu_\infty
\end{align*}
for
\begin{align*}
r'(c,\sigma):=\sum_{\psi_1 \in \hat G}\psi_1(\sigma^{-1})\psi_1(c)  \prod_{\psi_2 \in \hat G,\ \psi_2|_H=\psi_1|_H,\ \psi_2 \neq \psi_1}L_{dp}(0,\psi_2).
\end{align*}
Since $G/H=\mathrm{Gal}(F/\mathbb Q)$ is the cyclic group generated by $\mathrm{Frob}_p$, we have
\begin{align*}
\prod_{\psi_2 \in \hat G,\ \psi_2|_H=\psi_1|_H,\ \psi_2 \neq \psi_1}L_{dp}(0,\psi_2)
&=\prod_{i=1}^{[F:\mathbb Q]-1}(1-\psi_1(p)e^{\frac{2 \pi \sqrt{-1} i}{[F:\mathbb Q]}})\prod_{\psi_2 \in \hat G,\ \psi_2|_H=\psi_1|_H,\ \psi_2 \neq \psi_1}L_{d}(0,\psi_2) \\
&=\left(\sum_{i=0}^{[F:\mathbb Q]-1}\psi_1(p^i)\right)\prod_{\psi_2 \in \hat G,\ \psi_2|_H=\psi_1|_H,\ \psi_2 \neq \psi_1}L_{d}(0,\psi_2).
\end{align*}
It follows that
\begin{align*}
r'(c,\mathrm{Frob}_\mathfrak p\sigma)=\sum_{i=0}^{[F:\mathbb Q]-1} r([(p)^i]c,\mathrm{Frob}_\mathfrak p\sigma)=\sum_{i=1}^{[F:\mathbb Q]} r([(p)^{-i}]c,\sigma).
\end{align*}
Summarizing the above, we get
\begin{align*}
&[\frac{\tilde\tau\left(\prod_{c \in \phi_K^{-1}(\sigma)}  \exp(X(c;D,\mathfrak a_c))\right)}{\prod_{c \in \phi_K^{-1}(\sigma)}  \exp(X([\mathfrak p]c;D,\mathfrak p\mathfrak a_c))}:
\prod_{c \in \phi_K^{-1}(\sigma)} 
\frac{\pi_{\mathfrak p}^{\frac{\zeta(0,c)}{h_F^+}}}{\prod_{\tilde c \in \phi^{-1}( [\mathfrak p]c)}\exp_p(X_p(\tilde c;D,\mathfrak a_c))}]^{[F:\mathbb Q]} \\
&\equiv 
\prod_{c \in C_{(d)}}[\frac{\tilde\tau\left(\exp(\zeta'(0,c))\right)}{\exp(\zeta'(0,[(p)^{[F:\mathbb Q]}]c))}:
\frac{p^{\zeta(0,c)}}{\prod_{i=1}^{[F:\mathbb Q]}\left(\frac{p^{\zeta(0,[(p)^{i-1}]c)}}{\prod_{\tilde c \in \phi^{-1}([(p)^i]c)}\exp_p(\zeta'(0,\tilde c))}\right)}]^{\frac{r(c,\sigma)}{|H|}}.
\end{align*}
The same relation holds for the ``period part'' by (\ref{p1}), its $p$-adic version, and (\ref{p2}).
Hence we obtain 
\begin{align*}
\gamma(\sigma)\equiv \prod_{c \in C_{(d)}} \gamma(c)^{\frac{r(c,\sigma)}{|G|}}.
\end{align*}
Now the assertion follows from (\ref{irt}).
\end{proof}

\end{document}